\documentclass[11pt]{article}

\usepackage[english]{babel}
\newtheorem{theorem}{Theorem}
\newtheorem{lemma}{Lemma}
\newtheorem{definition}{Definition}
\newtheorem{proof}{Proof}
\usepackage{stmaryrd}
\usepackage{enumerate}
\usepackage{graphicx}
\usepackage{amssymb}
\usepackage{array}
\usepackage{url}
\usepackage{amsmath}
\usepackage{subcaption}
\usepackage[margin=1.0in]{geometry}
\DeclareMathAlphabet{\pazocal}{OMS}{zplm}{m}{n}

\usepackage{graphicx}

\usepackage[framemethod=tikz]{mdframed}
\usepackage{lipsum}

\definecolor{mycolor}{rgb}{0.122, 0.435, 0.698}
\definecolor{blue-green}{rgb}{0.0, 0.87, 0.87}
\definecolor{royalblue}{rgb}{0.01, 0.28, 1.0}
\definecolor{bluet}{rgb}{0.1, 0.1, 0.8}

\newmdenv[innerlinewidth=1pt, roundcorner=0.5pt,linecolor=mycolor,innerleftmargin=0.5pt,
innerrightmargin=0.5pt,innertopmargin=0.4pt,innerbottommargin=0.4pt]{mybox}

\DeclareMathAlphabet{\pazocal}{OMS}{zplm}{m}{n}

\usepackage{bm}

\newcommand{\m}{\mathbf}
\newcommand{\tra}{\top}
\newcommand{\wt}{\widetilde}
\newcommand{\wh}{\widehat}
\newcommand{\kp}{\kappa}
\newcommand{\s}{\sigma}
\newcommand{\w}{w}

\newcommand{\el}{\m{e}_1}
\newcommand{\er}{\m{e}_n}

\newcommand{\ev}{\delta}
\newcommand{\eu}{\gamma}
\newcommand{\thet}{\theta}

\newcommand{\ba}{\begin{array}}
\newcommand{\ea}{\end{array}}
\newcommand{\be}{\begin{equation}}
\newcommand{\ee}{\end{equation}}
\newcommand{\ben}{\begin{equation*}}
\newcommand{\een}{\end{equation*}}
\newcommand{\bd}{\begin{displaymath}}
\newcommand{\ed}{\end{displaymath}}
\newcommand{\bi}{\begin{itemize}}
\newcommand{\ei}{\end{itemize}}
\newcommand{\bn}{\begin{enumerate}}
\newcommand{\en}{\end{enumerate}}

\newcommand{\f}{\frac}

\newcommand{\ve}{\varepsilon}

\begin{document}

\title{Upwind summation-by-parts finite differences: error estimates and WENO methodology}


\author{
  Yan Jiang\thanks{School of Mathematical Sciences, University of Science and Technology of China, Hefei, 230026, Anhui, China.
  ({jiangy@ustc.edu.cn}).}
  \and Siyang Wang\thanks{Corresponding author. Department of Mathematics and Mathematical Statistics, Ume{\aa} University, Ume\aa,  Sweden
  ({siyang.wang@umu.se}).}}
\maketitle


\abstract{High order upwind summation-by-parts finite difference operators have recently been developed. When combined with the simultaneous approximation term method to impose boundary conditions, the method converges faster than using traditional summation-by-parts operators. We prove the convergence rate by the normal mode analysis for such methods for a class of hyperbolic partial differential equations. Our analysis shows that the penalty parameter for imposing boundary conditions affects the convergence rate for stable methods. In addition, to solve problems with discontinuous data, we extend the method to also have the weighted essentially nonoscillatory property.  The overall method is stable, achieves high order accuracy for smooth problems, and is capable of solving problems with discontinuities. 
}

\noindent \textbf{Keywords}: Finite difference methods, Summation-by-parts, Weighted essentially nonoscillatory methods, Conservation law, Stability, Accuracy

\noindent \textbf{ AMS}: 65M12

\section{Introduction}
By the classical dispersion analysis, high order methods are more efficient than low order method for solving wave propagation problems with sufficient smoothness  \cite{Hagstrom2012,Kreiss1972}. In many applications, wave propagation can be modeled by hyperbolic partial differential equations satisfying energy conservation or dissipation. For robustness of a numerical method, it is important that the discretization satisfies a corresponding discrete energy estimate. Such methods are often referred to as energy stable. When using finite difference methods to solve initial-boundary-value problems, a summation-by-parts (SBP) property is crucial to prove energy stability of the discretization. The SBP property discretely mimics the integration-by-parts principle, which is used to derive the continuous energy estimate.  In \cite{Kreiss1974}, Kreiss and Scherer introduced the SBP concept and constructed traditional SBP operators for the first derivative. These traditional SBP operators are based on central finite difference stencils in the interior, and special stencils on a few grid points near boundaries. The simultaneous-approximation-term (SAT) method is often used to weakly impose boundary conditions \cite{Carpenter1994}.  The SBP-SAT discretization has been derived and analyzed for a wide range of hyperbolic and parabolic problems, see the reviews \cite{Del2014Review,Svard2014}. 

Upwind SBP operators are constructed for the first derivative \cite{Mattsson2017}. These operators are based on biased finite difference stencils in the interior, and come in pairs for a given accuracy. When combining with the SAT method to impose boundary conditions, the upwind SBP-SAT discretization naturally introduces dissipation, which can lead to significantly smaller errors than the traditional SBP-SAT discretization without dissipation. As an example, it was observed in the numerical experiments in \cite{Mattsson2017}  that the upwind SBP-SAT method has a higher-than-expected convergence rate when solving the advection equation and a first-order linear hyperbolic system. The first main contribution of this paper is to derive error estimates by the normal mode analysis for a class of upwind SBP-SAT schemes. We prove that, with a particular choice of penalty parameters in the SAT method, the convergence rate is indeed higher than using traditional SBP operators.

For problems with piecewise smooth coefficients  modeling discontinuous media,  stable and high order multiblock SBP-SAT discretizations have been derived for hyperbolic and parabolic problems \cite{Almquist2019,Carpenter1999,Ludvigsson2018,Nordstrom1999}. The central idea is to decompose the domain into subdomains such that material discontinuities are aligned with the subdomain interfaces. An SBP finite difference discretization is constructed in each subdomain, and adjacent subdomains are patched together by the SAT method for physical interface conditions.  However, if the solution discontinuities move in time, the multiblock strategy cannot be used in a straightforward way. A standard finite difference discretization based on fixed stencils gives numerical solutions with oscillations near discontinuities. One technique to eliminate such spurious oscillations is the weighted essentially nonoscillatory (WENO)  method \cite{Jiang1996,Liu1994}. In a WENO scheme, a nonlinear adaptive procedure automatically selects the smoothest local stencil for avoiding stencils across discontinuities. The WENO schemes have been extensively developed to solve hyperbolic conservation laws. 
In recent years, a series of work on unequal-sized sub-stencils have been studied, improving the robustness and computational efficiency, as well as smaller residue convergences for steady-state problems \cite{ZHU2018659, ZHU201919, ZHU2020}.
We note that, however, energy stability can in general not be established for standard WENO schemes. 

Efforts have been made to develop schemes that have both the SBP and the WENO property, which is often referred to as energy stable WENO (ESWENO) method. An ESWENO method for the advection equation with periodic boundary conditions was developed in \cite{Yamaleev2009b,Yamaleev2009}. Based on traditional SBP operators and the SAT method, the ESWENO methodology was further developed to a fourth order accurate discretization for the advection equation with inflow boundary conditions \cite{Fisher2011}. Motivated by the superior accuracy property of upwind SBP finite difference discretization in \cite{Mattsson2017}, we extend the methodology and develop an upwind ESWENO method.

The rest of the paper is organized as follows. In Sec.~2, we introduce the upwind SBP property, and the SAT method for imposing boundary conditions. A discrete energy estimate is derived to ensure stability of the semidiscretization. In Sec.~3, we derive error estimates for the upwind SBP-SAT discretization using the normal mode analysis. In Sec.~4, we derive an ESWENO method based on the upwind SBP operators. Numerical examples are presented in Sec.~5 to verify the error estimates from Sec.~3 and the performance of the upwind ESWENO method from Sec.~4. Finally, we draw conclusion in Sec.~6.

\section{Spatial approximation}
\subsection{SBP property}

Consider a uniform grid $\m{x}$ in domain $[0,1]$ with grid points 
\begin{equation}\label{1d_grid}
x_j := (j-1)h,\quad j=1,\cdots,n, \quad h:=\frac{1}{n-1},
\end{equation}
where $n$ is the number of grid points and $h$ is the grid spacing. Central finite difference stencils can easily be constructed on the grid points away from the boundaries by using Taylor expansion. On the few grid points close to the boundaries, one-sided stencils are used as boundary closure. The boundary closure can be specially constructed so that the finite difference stencils satisfy a SBP property, which is a discrete analogue of the integration-by-parts principle. The SBP concept was first introduced by Kreiss and Scherer in \cite{Kreiss1974} on uniform grid \eqref{1d_grid}. We refer to this concept as \textit{traditional SBP}, and restate its definition below. 

\begin{definition}[traditional SBP]
The difference operator $D$ is a first derivative SBP operator if it satisfies 
\[
D = H^{-1} \left(Q+\frac{B}{2}\right),
\]
where $H$ is symmetric positive definite, $Q=-Q^\tra$ and $B=\text{diag}([-1,0,0,\cdots,0,0,1])$.
\end{definition}

The operator $H$ defines a discrete inner product, thus a discrete norm $\|\cdot\|_H$. It is also a quadrature \cite{Hicken2016}. 
In a traditional SBP operator, a standard central finite difference stencil is used on the grid points  in the interior; on a few grid points close to boundaries, special boundary closures are constructed so that the SBP property is satisfied. With a $p^{th}$ order accurate ($p$ is even) central finite difference stencil in the interior, the boundary closures can at most be $(p/2)^{th}$ order accurate. In \cite{Kreiss1974}, the SBP  operators are constructed for $p=2,4,6,8$ with boundary closures on the first $p/2$  grid points. It is important to note that the number of grid points with the boundary closures is independent of grid spacing $h$. 
When using the traditional SBP operator to solve first order hyperbolic problems, the convergence rate is in general $p/2+1$ \cite{Gustafsson1975,Gustafsson1981}.

Recently, Mattsson {\color{black}constructed} upwind SBP operators \cite{Mattsson2017}. Comparing with the traditional SBP operators, the main difference is that noncentered finite difference stencils are used on the grid points in the interior. 
Upwind SBP operators with $p$ from 2 to 9 were constructed in  \cite{Mattsson2017}, and for each $p$ the operators come in pairs, with one operator skewed to the left and the other operator skewed to the right. 
\begin{definition}[upwind SBP]\label{def_upwind}
The difference operator $D_m$ and $D_p$ are a pair of first derivative upwind SBP operators if they can be written as 
\[
D_m = H^{-1} \left(Q_m+\frac{B}{2}\right),\quad D_p = H^{-1} \left(Q_p+\frac{B}{2}\right),
\]
where $Q_m=-Q_p^\tra$, $Q_m+Q_m^\tra$ is symmetric positive semidefinite, and $B=\text{diag}([-1,0,0,\cdots,0,0,1])$.
\end{definition}
Similar to the traditional SBP operators, the accuracy near boundaries is lower than in the interior. With a $p^{th}$ order accurate  interior stencil, the accuracy of the boundary closures is $(p-1)/2$ on the first $(p+1)/2$ grid points for odd $p$, and $p/2$ on the first $p/2$ grid points for even $p$. The subscripts $m$ and $p$ are abbreviations of \textit{minus} and \textit{plus}, corresponding to biased stencils to the left and right directions, respectively. 

\subsection{An SBP-SAT discretization}
We consider the linear advection equation as our model problem,
\begin{equation}\label{eqn_adv}
U_t + U_x = 0,\quad x\in [0,1],\ t\in [0,1],
\end{equation}
with  inflow boundary condition $U(0,t) = g(t)$. In this case, the initial profile of $U$ is transported from left to right, and the upwind direction is to the left.    

 We discretize the equation in space on a uniform grid \eqref{1d_grid}. Let $\m{u}=[u_1, u_2, \cdots, u_n]^T$ be the finite difference solution, where $u_j\approx U(x_j,t)$. The SBP-SAT discretization of \eqref{eqn_adv} reads
\begin{equation}\label{semi_adv}
\m{u}_t + D_m \m{u} = \tau H^{-1}\m{e_0} (u_1-g(t)),
\end{equation}
where $\m{e_0}:=[1,0,0,\cdots,0]^\tra$. The term on the right-hand side is a penalty term weakly imposing the boundary condition $U(0,t) = g(t)$, and the penalty parameter $\tau$ is determined in the stability analysis so that \eqref{semi_adv} satisfies a discrete energy estimate. 

\begin{lemma}\label{lemma_sta}
The semidiscretization \eqref{semi_adv} with $g(t)\equiv 0$ satisfies the energy estimate $\f{d}{dt}\|\m{u}\|_H^2\leq 0$ if the penalty parameter $\tau\leq -1/2$.
\end{lemma}
\begin{proof}
We multiply \eqref{semi_adv} by $\m{u}^\tra H$ on both sides, and obtain
\[
\m{u}^\tra H \m{u}_t + \m{u}^\tra H D_m \m{u} = \tau (u_1)^2.
\]
Adding the above equation to its tranpose, we have 
\begin{align*}
\f{d}{dt}\|u\|_H^2 &= -\m{u}^\tra  (H D_m+(H D_m)^\tra) \m{u} +  2\tau (u_1)^2 \\
&=-\m{u}^\tra  (Q_m+Q_m^\tra+B) \m{u} +  2\tau (u_1)^2\\
&=-\m{u}^\tra  (Q_m+Q_m^\tra) \m{u} +(2\tau+1) (u_1)^2- (u_n)^2\\
&\leq 0,
\end{align*}
if $\tau\leq -1/2$. We have used  $\m{u}^\tra  (Q_m+Q_m^\tra) \m{u} \geq 0$ by Definition \ref{def_upwind}.
\end{proof}
We call such a semidiscretization energy stable. 
In the above energy analysis,  $-\m{u}^\tra  (Q_m+Q_m^\tra) \m{u}$ is an artificial dissipation term. The same analysis carries over to the case with traditional SBP operators but without any artificial dissipation.

The convergence rate in $L^2$ norm of the traditional SBP-SAT discretization of the advection equation is $p/2+1$, i.e. one order higher than the order of truncation error of the boundary closure. This phenomenon, often termed as \textit{gain in convergence}, was theoretically analyzed in \cite{Gustafsson1975,Gustafsson1981} using the normal mode analysis. For the upwind SBP-SAT discretization, it is observed in \cite{Mattsson2017} that the \textit{gain in convergence} is more than one order. In the following section, we prove that for the advection equation, with most of the choices of the penalty parameter $\tau\leq -1/2$ the \textit{gain in convergence} is one order, i.e. the same as the traditional SBP-SAT discretization; the special choice of $\tau = -1$ results in a \textit{gain in convergence} of order 1.5. In addition, we also consider a first order hyperbolic system and prove that the \textit{gain in convergence} is 1.5 {\color{black} when the penalty parameter and the boundary condition satisfy a specific relation, while all other stable choices give a \textit{gain in convergence} of one order}.

\section{Accuracy analysis for the linear schemes} \label{sec_accuracy}
We derive error estimates for the linear schemes for both the scalar advection equation and a first order hyperbolic system. We assume the initial and boundary data are compatible and sufficiently smooth. This compatibility condition allows a decomposition of equation \eqref{eqn_adv} in a bounded domain to two half-line problems, in $x\in (-\infty, 1]$ 
and $x\in [0, \infty)$, and the Cauchy problem in $x\in (-\infty, \infty)$ \cite{Gustafsson2013}. We discretize the equation and prove stability in the bounded domain, but analyze the accuracy of discretization for the three decomposed problems.     

\subsection{Error estimate for the linear scheme for the advection equation}\label{sec_err_adv}

Let $\m{v}$ be the exact solution of \eqref{eqn_adv} evaluated on the grid. Then $\m{v}$ satisfies
\begin{equation}\label{exac_adv}
\m{v}_t + D_m \m{v} = \m{d},
\end{equation}
where $\m{d}$ is the truncation error on the grid $\m{x}$. The truncation error has the following structure 
\[
 \m{d} = [\underbrace{\mathcal{O}(h^b),\cdots, \mathcal{O}(h^b)}_{r}, \mathcal{O}(h^p), \cdots\cdots\cdots, \mathcal{O}(h^p),\underbrace{\mathcal{O}(h^b), \cdots,\mathcal{O}(h^b)}_{r}]^\tra,
\]
where $b<p$. 
In the analysis, we decompose the truncation error into three parts $\m{d}:=\m{d^{(\bm{\ve})}}+\m{d^{(I)}}+\m{d^{(\bm{\xi})}}$.  The first $r$ elements of the left boundary truncation error $\m{d^{(\bm{\ve})}}$ are  $\mathcal{O}(h^b)$ and the remaining elements in  $\m{d^{(\bm{\ve})}}$ are all zeros. Similarly in the right boundary truncation error $\m{d^{(\bm{\xi})}}$, the last $r$ elements $\mathcal{O}(h^b)$ are the only nonzero elements. In $\m{d^{(I)}}$, the first and last $r$ elements are zero but the elements in the interior are  $\mathcal{O}(h^p)$. 
It is important to note that $r$ does not depend on $h$, and the number of grid points with truncation error $\mathcal{O}(h^p)$ is proportational to $h^{-1}$.

Let $\bm{\zeta}:=\m{v}-\m{u}$ be the pointwise error. The error equation reads
\begin{equation}\label{err1_adv}
\bm{\zeta}_t + D_m \bm{\zeta} = -\tau H^{-1}\m{e_0} (\zeta_1)+\m{d}.
\end{equation}
The truncation error $\m{d}$ is a forcing term. Applying the energy analysis to the error equation \eqref{err1_adv}, it is straightforward to obtain the estimate $\|\bm{\zeta}\|\leq Ch^{\min(p,b+1/2)}$. The half an order gain from the boundary truncation error is due to the fact that $r$ does not depend on $h$.

Since the scheme is linear, we also decompose the pointwise error $\bm{\zeta}:=\bm{\ve}+\bm{\zeta}^{\m{(I)}}+\bm{\xi}$, where the three components $\bm{\ve}$, $\bm{\zeta}^{\m{(I)}}$ and $\bm{\xi}$ correspond to the truncation error $\m{d^{(\bm{\ve})}}$, $\m{d^{(I)}}$ and $\m{d^{(\bm{\xi})}}$, respectively. In the following, we analyze the three components of the pointwise error, and then combine them to obtain an error estimate for $\bm{\zeta}$.

An estimate for the interior pointwise error $\bm{\zeta}^{\m{(I)}}$ can be established by the energy analysis. Following the steps in the proof of Lemma \ref{lemma_sta}, we have 
\[
\f{d}{dt}\|\bm{\zeta}^{\m{(I)}}\|_H^2 \leq  2 (\bm{\zeta}^{\m{(I)}})^\tra H \m{d^{(I)}} \leq 2 \|\bm{\zeta}^{\m{(I)}}\|_H \|\m{d^{(I)}}\|_H.
\]
Consequently,
\[
\f{d}{dt}\|\bm{\zeta}^{\m{(I)}}\|_H  \leq  \|\m{d^{(I)}}\|_H.
\]
An integration in time gives
\[
\|\bm{\zeta}^{\m{(I)}}\|_H  \leq  Ch^p,
\]
where $C$ depends on the exact solution and the final time but not $h$. {\color{black}In the following, we use $C$ to denote a generic constant that may depend on the continuous problem but not $h$, and may take different values in different estimates. }

The above approach can be used to analyze the boundary pointwise error, leading to the estimates
\[
\|\bm{\ve}\|_H  \leq  Ch^{b+1/2},\quad \|\bm{\xi}\|_H \leq  Ch^{b+1/2}.
\]
However, these results are suboptimal, and sharper error estimates can be obtained by the normal mode analysis based on Laplace transform. 

Now consider the error equation for $\bm{\ve}$, 
\begin{equation}\label{err2_adv}
\bm{\ve}_t + D_m \bm{\ve} = -\tau H^{-1}\m{e_0} (\ve_1)+\m{d^{(\bm{\ve})}}.
\end{equation}
After Laplace transforming \eqref{err2_adv} in time and multiplying with $h$, we obtain
\begin{equation}\label{err3_adv}
\wt{s}\wh{\bm{\ve}} + hD_m \wh{\bm{\ve}} = -\tau hH^{-1}\m{e_0} (\wh{\ve_1})+h\wh{\m{d^{(\bm{\ve})}}},
\end{equation}
where $\wt s:=sh$ with $s$ the dual of time, and $\wh{\bm{\ve}}$ is the Laplace-transformed variable of $\bm{\ve}$. The central idea for  sharp error estimates is first deriving an error bound for $\bm{\ve}$ in the Laplace space, and then use Parseval's relation to obtain the corresponding estimate in the physical space. To this end, we shall solve the difference equation \eqref{err3_adv} in the following two steps. 

In the interior, the upwind SBP operator $D_m$ uses a repeated finite difference stencil, and the components on the right-hand side of \eqref{err3_adv} equal to zero. More precisely, we have 
\begin{equation}\label{err4_adv}
\wt s\wh{\ve}_j + h\sum_{i=j-k_1}^{j+k_2}(D_m)_{ji}\wh{\ve}_i = 0,\quad j=r+1, r+2, \cdots,
\end{equation}
where $k_1=(p-1)/2$, $k_2=(p+1)/2$ for odd $p$, and $k_1=(p-2)/2$, $k_2=(p+2)/2$ for even $p$. The solution to \eqref{err4_adv} can be obtained by solving the corresponding characteristic equation. 

In the first $r$ equations in \eqref{err3_adv}, the SBP operator $D_m$ has boundary modified stencils and the truncation error has nonzero components. In this case, a system of $r$ linear equations shall be solved and its solution depends on $\wt s$. It is important to analyze the behavior of the solution in the limit $\wt s\rightarrow 0$, because the convergence rate corresponds to the asymptotic regime when $h\rightarrow 0$. 

The stencils in $D_m$ and the SAT for the inflow boundary condition affect the solution to \eqref{err3_adv}. Therefore, to obtain error estimates with 
 $D_m$ of different order of accuracy, the analysis must be performed for each $D_m$. 
In \cite{Mattsson2017}, upwind SBP operators $D_m$ with interior order $p=2$ to 9 are constructed. We present below the accuracy property of the semidiscretization with $p=3$ and its proof.

\begin{theorem}\label{thm_adv_err}
Consider the stable semidiscretization \eqref{semi_adv} with a third order accurate interior stencil and first order accurate boundary closure in $D_m$. The convergence rate is 2.5 if the penalty parameter is $\tau=-1$, and 2 if $\tau\neq -1$.
\end{theorem}

\begin{proof}
We start by considering $\bm\ve$ corresponding to the inflow boundary condition at $x=0$. 
The interior error equation in {\color{black}Laplace} space reads
\[
\wt s \wh \ve_j = -\f{1}{6}\wh \ve_{j-2}+\wh \ve_{j-1}-\f{1}{2}\wh \ve_{j}-\f{1}{3}\wh \ve_{j+1}, \quad j = 3,4,\cdots.
\]
The corresponding characteristic equation is 
\[
-\f{1}{6}+\kp-\left(\f{1}{2}+\wt s\right)\kp^2-\f{1}{3}\kp^3=0\quad j = 3,4,\cdots.
\]
At $\wt s=0$, the characteristic equation has three solutions 
\[
1,\quad \frac{-5+\sqrt{33}}{4}\approx 0.1861,\quad \frac{-5-\sqrt{33}}{4}\approx -2.6861.
\]
By a perturbation analysis, the admissible solutions satisfying $|\kp|<1$ are 
\[
\kp_1=1-\wt s,\quad \kp_2=\frac{-5+\sqrt{33}}{4}+\mathcal{O}(\wt s)\approx 0.1861+\mathcal{O}(\wt s).
\]
We note that $\kp_1$ is a slowly decaying solution and $\kp_2$ is a fast decaying solution. The general solution to the interior error equation can be written as 
\begin{equation}\label{g_sol}
\wh\ve_j = \s_1\kp_1^{j-1}+\s_2\kp_2^{j-1},\quad j=1,2,\cdots,
\end{equation}
where the coefficients $\s_1$ and $\s_2$ will be determined by the error equation near the boundary. 

On the first two grid points, the error equation \eqref{err3_adv} with $j=1,2$ is 
\begin{align*}
\wt s \wh\ve_1 - \wh\ve_1+\wh\ve_2-\f{12}{5}\tau\wh \ve_1 &= \f{1}{2}h^2\wh v_{xx}(x_1), \\
\wt s \wh\ve_2 - \f{9}{13}\wh \ve_1+\f{5}{13}\wh\ve_2+\f{4}{13}\wh\ve_3&= -\f{5}{26}h^2\wh v_{xx}(x_1), 
\end{align*}
Using the general solution \eqref{g_sol}, we obtain the boundary system  $C_3\bm{\Sigma}=\m{d_3}$, where 
\[
C_3=\begin{bmatrix}
\wt s-1-\f{12}{5}\tau+\kp_1 & \wt s-1-\f{12}{5}\tau+\kp_2 \\
-\f{9}{13}+\wt s\kp_1+\f{5}{13}\kp_1+\f{4}{13}\kp_1^2 & -\f{9}{13}+\wt s\kp_2+\f{5}{13}\kp_2+\f{4}{13}\kp_2^2
\end{bmatrix},\]
\[
\bm{\Sigma}=\begin{bmatrix}
\s_1\\
\s_2
\end{bmatrix},\quad
\m{d_3}=\begin{bmatrix}
\f{1}{2}\\
-\f{5}{26}
\end{bmatrix}h^2\wh v_{xx}(x_1).
\]
At $\wt s=0$, we compute the determinant of the boundary system $\det(C_3)= 3\tau (3+5\sqrt{33})/65$, which is nonzero for all $\tau\leq -1/2$ required by stability. Thus, the determinant condition is satisfied. The solution to the boundary system is 
\begin{equation}\label{bs_sol}
\begin{bmatrix}
\s_1\\
\s_2
\end{bmatrix}=\begin{bmatrix}
-\f{10(1+\tau)}{(3+5\sqrt{33})\tau}\\
\f{10}{(3+5\sqrt{33})}
\end{bmatrix}h^2\wh v_{xx}(x_1).
\end{equation}
 The choice $\tau=-1$ gives $\s_1=0$. This is significant, because $\s_1$ is multiplied with the slowly decaying solution $\kp_1$ in \eqref{g_sol}. In this case, substituting \eqref{bs_sol} to \eqref{g_sol}, we obtain
 \begin{equation}\label{eps_err0}
 \|\bm{\wh\ve}\|_h = \sqrt{h\sum_{j=1}^\infty |\s_2\kp_2^{j-1}|^2} = \sqrt{h\s_2^2 \f{1}{1-|\kp_2^2|}} \leq Ch^{2.5} |\wh v_{xx}(x_1)|\quad \text{when } \tau=-1.
\end{equation}
Even though the truncation error at the inflow boundary condition is $\mathcal {O}(h)$, the corresponding error  converges to zero with rate 2.5 when $\tau=-1$. 

When $\tau\neq -1$, the convergence rate is lower. To see this, we compute
 \begin{align}
 \|\bm{\wh\ve}\|_h &= \sqrt{h\sum_{j=1}^\infty |\s_1\kp_1^{j-1}+\s_2\kp_2^{j-1}|^2} \leq \sqrt{h\sum_{j=1}^\infty |\s_1\kp_1^{j-1}|^2}+\sqrt{h\sum_{j=1}^\infty\s_2|\kp_2^{j-1}|^2}\label{eps_err1}\\
 &\leq \sqrt{h|\s_1|^2 \f{1}{1-|\kp_1^2|}}+\sqrt{h|\s_2|^2 \f{1}{1-|\kp_2^2|}}\label{eps_err2}
\end{align}
For a bound on the first term of \eqref{eps_err2}, we need to use Lemma 2 in \cite{Wang2017}. For completeness, we restate the lemma below. 
\begin{lemma}[Lemma 2 in \cite{Wang2017}] \label{lemma_kp}
Consider $Re(\wt s) := \eta h > 0$, where $\eta$ is a constant independent of $h$. Then $\kp_1=1-\wt s +\mathcal{O}({\wt s}^2)$ satisfies
\[
\f{1}{1-|\kp_1^2|}\leq\f{1}{2\eta h}
\]
to the leading order.
\end{lemma}
Combining \eqref{eps_err1}-\eqref{eps_err2} and Lemma \ref{lemma_kp}, we have 
\[
 \|\wh{\bm\ve}\|_h \leq C h^2 + C h^{2.5} \leq C h^2|\wh v_{xx}(x_1)|\quad  \text{when } \tau\neq -1.
 \]
Therefore, the error in  converges to zero with rate 2 when $\tau\neq -1$, which is half an order lower than when $\tau=-1$.

The above estimates in Laplace space can be transformed to the corresponding error estimates in physical space  by Parseval's relation. For the case with $\tau=-1$, we have 
\begin{align*}
\int_0^\infty e^{-2\eta t}\|\bm{\ve}\|_h^2 dt &= \f{1}{2\pi}\int_{-\infty}^ {\infty} \|\wh{\bm{\ve}}(\eta+iy)\|_h^2 dy \\
&\leq \f{ C h^5}{2\pi} \int_{-\infty}^ {\infty} |\wh v_{xx}(x_1,\eta+iy)|{\color{black}^2} dy \\
&= C h^5 \int_{0}^ {\infty} e^{-2\eta t} |v_{xx}(x_1,t)|{\color{black}^2} dt.
\end{align*}
By the standard argument of \textit{future cannot affect past} \cite[pp.~294]{Gustafsson2013}, the integration domain $[0,\infty)$ can be changed to any finite final time $t_f$, resulting 
\[ 
\|\bm{\ve}\|_{h,t}:=\sqrt{\int_0^{t_f}\|\bm{\ve}\|_h^2 dt} \leq C h^{2.5}\sqrt{  \int_{0}^ {t_f} e^{2\eta t_f} |v_{xx}(x_1,t)|{\color{black}^2} dt}.
\]
The estimate in physical space for the case with $\tau\neq -1$ can be obtained analogously. 

We now consider the outflow boundary $x=1$. Here, no boundary condition is specified and there is no corresponding  SAT in the semidiscretization. However, on the last two grid points the stencils in $D_m$ are also modified. The characteristic equation is 
\[
\f{1}{3}+\left(\f{1}{2}+\wt s\right)\kp-\kp^2+\f{1}{6}\kp^3=0,
\]
and the only admissible root satisfying $|\kappa|<1$ for $Re(\wt s)>0$ is 
\[
\kp_3 = \f{5-\sqrt{33}}{2}+\mathcal{O}(\wt s),
\]
which is a fast decaying component. Let the error $\bm{\xi}:=[\cdots, \xi_2, \xi_1]^\tra$, then the error equation on the last two grid points are 
\begin{align*}
  \wt s \wh\xi_2 + \f{2}{13}\wh\xi_4-\f{12}{13}\wh\xi_3+\f{5}{13}\wh\xi_2+\f{5}{13}\wh\xi_1 &= -\f{1}{26}h^2\wh v_{xx}(x_n), \\
  \wt s \wh\xi_1 + \f{2}{5}\wh\xi_3-\f{9}{5}\wh\xi_2+\f{7}{5}\wh\xi_1 &= \f{1}{10}h^2\wh v_{xx}(x_n),
\end{align*}
with two unknown variables $\wh\xi_1$ and $\wh\xi_2$. Note that $\xi_j = \xi_2\kp_3^{j-2}$ for $j=3,4,\cdots$. At $\wt s=0$, these two equations have a unique solution, i.e. the determinant condition is satisfied. Therefore, we have 
\[
 \|\wh{\bm\xi}\|_h = \sqrt{h|\wh\xi_1|^2+h\sum_{j=2}^\infty |\wh\xi_j|^2}=\sqrt{h|\wh\xi_1|^2+h|\wh\xi_2|^2\f{1}{1-|\kp_3|^2}}\leq Ch^{2.5}|\wh v_{xx}(x_n)|.
 \]
The error $\|\wh{\bm\xi}\|_h$ converges to zero with rate 2.5, even though the truncation error on the last two grid points is $\mathcal{O}(h)$. By using Parseval's relation, we have 
\[
 \|{\bm\xi}\|_{h,t} \leq Ch^{2.5}\sqrt{   \int_{0}^ {t_f} e^{2\eta t_f} |v_{xx}(x_n,t)|{\color{black}^2} dt}.
\]
Combining the estimates for the two boundary truncation error $\bm{\ve}$ and $\bm\xi$, the error estimate for $\bm\zeta$ follows. This concludes the proof.
\end{proof}

When using a finite difference method to solve an initial boundary value problem, it is common that the truncation error is larger near boundaries than in the interior. The general results in \cite{Gustafsson1975,Gustafsson1981} state that for first order hyperbolic problems, the convergence rate can be one order higher than the boundary truncation error, i.e. there is a gain in convergence of one order. In Theorem \ref{thm_adv_err}, we have proved that for most choices $\tau\leq -1/2$, $\tau\neq -1$, the gain in convergence is indeed one order, resulting in an overall convergence rate of order 2. With the particular choice $\tau=-1$, a higher convergence rate of order 2.5 is obtained, corresponding to a gain in convergence of order 1.5.  Hence, the penalty parameter $\tau=-1$ shall be used in practical computation. 

\subsection{Error estimate for the linear scheme for a first order hyperbolic system}\label{sec_accuracy_system}
Consider the first order hyperbolic system 
\begin{equation}\label{hs}
\m{U}_t + A\m{U}_x = 0,\quad x\in [0,1],\ t\in [0,1],
\end{equation}
where the $2\times 2$ coefficient matrix is 
\begin{equation}\label{A}
A = \begin{bmatrix}
0 & 1\\
1 & 0
\end{bmatrix}, 
\end{equation}
and the unknown $\m{U}(x,t):=[U(x,t), V(x,t)]^\tra$ has two components. There are several boundary conditions that lead to a wellposed problem. Here, we consider 
\begin{equation}\label{hs_bc}
{\color{black}U(0,t)+\alpha_0 V(0,t)=g_1(t),\quad U(1,t)+\alpha_1 V(1,t)=g_n(t).}
\end{equation}

{\color{black}
\begin{theorem}
  The first order hyperbolic system \eqref{hs}-\eqref{A} with boundary condition \eqref{hs_bc} satisfies an energy estimate if $\alpha_0\geq 0$ and $\alpha_1\leq 0$. 
\end{theorem}

\begin{proof}
    We multiply \eqref{hs} by $\m{U}^\tra$, and obtain 
    \begin{align*}
        \frac{d}{dt}\int_0^1 U^2+V^2dx = -2UV|_0^1 = -2\alpha_0 V^2(0,t) + 2\alpha_1 V^2(1,t)
    \end{align*}
with homogeneous boundary condition $g_1=g_n=0$. Therefore, we have the energy estimate $\frac{d}{dt}\int_0^1 U^2+V^2dx \leq 0$ when $\alpha_0\geq 0$ and $\alpha_1\leq 0$. 
\end{proof}
}

We discretize the equation in space on a uniform grid $\m{x}$ defined in \eqref{1d_grid}. Let $\m\w:=[\m{u},\m{v}]^\tra$ be the finite difference solution such that $u_j\approx U(x_j,t)$ and $v_j\approx V(x_j,t)$. To derive a stable finite difference scheme, we perform a flux splitting 
\begin{equation*}
A := A_m+A_p := \frac{1}{2}\begin{bmatrix}
1 & 1\\
1 & 1
\end{bmatrix}+\frac{1}{2}\begin{bmatrix}
-1 & 1\\
1 & -1
\end{bmatrix}, 
\end{equation*}
where $A_m$ corresponds to the wave going from left to right, and $A_p$ corresponds to the wave going from right to left. We write the semidiscretization as 
{\color{black}
\begin{align}\label{semi_hs}
&\m\w_t + (A_p \otimes D_p) \m\w + (A_m \otimes D_m)\m\w \\
=&\begin{bmatrix}
\tau_1 H^{-1} \el (u_1+\alpha_0 v_1-g_1) \\
\tau_2 H^{-1} \el (u_1+\alpha_0 v_1-g_1)
\end{bmatrix}+\begin{bmatrix}
\tau_3 H^{-1} \er (u_n+\alpha_1 v_n-g_n) \\
\tau_4 H^{-1} \er (u_n+\alpha_1 v_n-g_n)
\end{bmatrix}.\notag
\end{align}}
The last two terms on the left-hand side of \eqref{semi_hs} approximate $A\m{U}_x$ using the pair of upwind SBP operators, while the right-hand side imposes weakly the Dirichlet boundary conditions. The four penalty parameters $\tau_1$, $\tau_2$, $\tau_3$, $\tau_4$ are determined by the following stability analysis.

\begin{lemma}
{\color{black}The semidiscretization \eqref{semi_hs} satisfies the energy estimate $\frac{d}{dt}\|\m\w\|_H\leq 0$ 
 if 
 \begin{equation}\label{StabilityCondition}
 (\alpha_0\tau_1-\tau_2-1)^2+4\alpha_0\tau_1\leq 0,\quad (\alpha_1\tau_3-\tau_4+1)^2-4\alpha_1\tau_3\leq 0.
 \end{equation}}
\end{lemma}
{\color{black}
\begin{proof}
The proof by the energy analysis follows the same principle as the proof for Lemma \ref{lemma_sta}, see also Lemma 3.1 in \cite{Mattsson2017}. More precisely, for homogeneous boundary data the energy change rate is 
\begin{align*}
\frac{d}{dt}\|\m\w\|_H^2 \leq &
\m\w^\tra (I \otimes (Q_p+Q_p^\tra))\m\w \\
&+2u_1v_1 + 2\tau_1(u_1+\alpha_0 v_1) + 2\tau_2 v_1(u_1+\alpha_0 v_1) \\
&-2u_nv_n+2\tau_3 u_n(u_n+\alpha_1 v_n) + 2\tau_4 v_n(u_n+\alpha_1 v_n).
\end{align*}
On the right-hand side, the first term is nonpositive because of the SBP property. The second line corresponds to the boundary terms at $x=0$ and the third line corresponds to the boundary terms at $x=1$, and they can be analyzed separately.  

Consider the boundary terms at $x=0$, i.e., $BT_0:=2u_1v_1 + 2\tau_1 u_1 (u_1+\alpha_0 v_1) + 2\tau_2 v_1(u_1+\alpha_0 v_1)$. If $\tau_1=0$, then we must choose $\tau_2=-1$ so that $BT_0=-2\alpha_0 v_1^2\leq 0$. If $\tau_1\neq 0$, then we have 
\begin{align*}
BT_0=2\tau_1 \left[\right.  \left(u_1+\right.\left(\frac{\alpha_0}{2}+\frac{\tau_2+1}{2\tau_1}\right)v_1\left.\right)^2 - \left(\frac{\alpha_0}{2}+\frac{\tau_2+1}{2\tau_1}\right)^2v_1^2+\frac{\tau_2\alpha_0}{\tau_1}v_1^2      \left.\right].
\end{align*}
We can guarantee $BT_0\leq 0$ if $\tau_1<0$ and
\[
- \left(\frac{\alpha_0}{2}+\frac{\tau_2+1}{2\tau_1}\right)^2+\frac{\tau_2\alpha_0}{\tau_1}\geq 0 \Rightarrow (\alpha_0\tau_1-\tau_2-1)^2+4\alpha_0\tau_1\leq 0.
\]
The boundary terms at $x=1$ can be analyzed in a similar way, yielding the stability condition $(\alpha_1\tau_3-\tau_4+1)^2-4\alpha_1\tau_3\leq 0$.  
\end{proof}
}

To derive error estimates, we follow the main steps of the accuracy analysis for the advection equation by the normal mode analysis. First, we need to use the change of variables to the error equation of the hyperbolic system. We then split the error component into an interior part and a boundary part. The interior part can be estimated by the energy method in the same way as for the advection equation in Sec.~\ref{sec_err_adv}. It is the error due to the boundary closure that dominates the overall convergence rate. We have the following theorem for the error estimate.

\begin{theorem}
{\color{black}
Consider the stable semidiscretization \eqref{semi_hs} with a third order accurate interior stencil and first order accurate boundary closure in the pair $D_m$ and $D_p$. The convergence rate is 2.5 if 
\begin{equation}\label{AlphaTau}
\alpha_0 \tau_1 +\tau_2+1=0,\quad \alpha_1 \tau_3 +\tau_4-1=0.     
\end{equation}
Otherwise, the convergence rate is 2.}
\end{theorem}

\begin{proof}

We start by considering the error component $\bm{\ve}:=[\bm\eu,\bm\ev]^\tra$ that is due to the truncation error near the left boundary $x=0$ when using the upwind operators with third order accurate interior stencil.

 In the interior, we have the error equation in Laplace space
\begin{align}
\wt s \wh{\eu}_j &= -\frac{1}{12}\wh{\eu}_{j-2}+\frac{1}{3}\wh{\eu}_{j-1}-\frac{1}{2}\wh{\eu}_{j}+\frac{1}{3}\wh{\eu}_{j+1}-\frac{1}{12}\wh{\eu}_{j+2}-\frac{1}{12}\wh{\ev}_{j-2}+\frac{2}{3}\wh{\ev}_{j-1}-\frac{2}{3}\wh{\ev}_{j+1}+\frac{1}{12}\wh{\ev}_{j+2},\label{es1}\\
\wt s \wh{\ev}_j &= -\frac{1}{12}\wh{\eu}_{j-2}+\frac{2}{3}\wh{\eu}_{j-1}-\frac{2}{3}\wh{\eu}_{j+1}+\frac{1}{12}\wh{\eu}_{j+2} -\frac{1}{12}\wh{\ev}_{j-2}+\frac{1}{3}\wh{\ev}_{j-1}-\frac{1}{2}\wh{\ev}_{j}+\frac{1}{3}\wh{\ev}_{j+1}-\frac{1}{12}\wh{\ev}_{j+2}.\label{es2} 
\end{align}
Next, we compute the addition \eqref{es1}+\eqref{es2} and subtraction \eqref{es1}-\eqref{es2}, 
\begin{align}
\wt s (\wh{\eu}_j+\wh{\ev}_j) &= -\frac{1}{6}(\wh{\eu}_{j-2}+\wh{\ev}_{j-2})+(\wh{\eu}_{j-1}+\wh{\ev}_{j-1})-\frac{1}{2}(\wh{\eu}_{j}+\wh{\ev}_{j})-\frac{1}{3}(\wh{\eu}_{j+1}+\wh{\ev}_{j+1}), \label{es3}\\
\wt s (\wh{\eu}_j-\wh{\ev}_j) &= -\frac{1}{3}(\wh{\eu}_{j-1}-\wh{\ev}_{j-1})-\frac{1}{2}(\wh{\eu}_{j}-\wh{\ev}_{j})+(\wh{\eu}_{j+1}-\wh{\ev}_{j+1})-\frac{1}{6}(\wh{\eu}_{j+2}-\wh{\ev}_{j+2}). \label{es4}
\end{align}
which can be considered as relations for $\bm\eu+\bm\ev$ and $\bm\eu-\bm\ev$, respectively. The characteristic equation corresponding to \eqref{es3} is 
\begin{equation*}
\wt s \kp^2 =  -\frac{1}{6} + \kp - \frac{1}{2}\kp^2 -\frac{1}{3}\kp^3,
\end{equation*}
which has two admissible roots $\kp_1 = 0.1861+\mathcal{O}(\wt s)$ and $\kp_2 = 1-\wt s+\mathcal{O}(\wt s^2)$. Similarly, the characteristic equation corresponding to \eqref{es4} is 
\begin{equation*}
\wt s \thet = -\frac{1}{3}-\frac{1}{2}\thet + \thet^2 -\frac{1}{6}\thet^3,
\end{equation*}
and has one admissible root $\thet = -0.3723+\mathcal{O}(\wt s)$. The general solutions to the error equations \eqref{es3} and \eqref{es4} are
\begin{align*}
\wh{\eu}_{j}+\wh{\ev}_{j} &= \sigma_1 \kp_1^{j-1}+\sigma_2 \kp_2^{j-1},\quad j = 1, 2,\cdots,\\
\wh{\eu}_{j}-\wh{\ev}_{j} &= \sigma_0 \thet^{j-2},\quad j = 2,3,\cdots.
\end{align*}
We then have
\begin{align}
\wh{\eu}_{j} &=\frac{1}{2} (\sigma_0 \thet^{j-2}+\sigma_1 \kp_1^{j-1}+{\color{black}\sigma_2} \kp_2^{j-1}),\quad j = 2,3,\cdots, \label{es5}\\
\wh{\ev}_{j} &= \frac{1}{2} (-\sigma_0 \thet^{j-2}+\sigma_1 \kp_1^{j-1}+\sigma_2 \kp_2^{j-1}),\quad j = 2,3,\cdots, \label{es6}\\
\wh{\ev}_{1}  &= -\wh{\eu}_{1}+ \sigma_1 +\sigma_2. \label{es7}
\end{align}
The four unknowns $\sigma_0$, $\sigma_1$, $\sigma_2$, and $\wh{\eu}_{1}$ shall be determined by the boundary closure. To this end, we consider the error equation on the first two grid points 
\begin{align*}
\wt s \wh{\eu}_{1} &= -\f{1}{5}\wh{\eu}_{1} + \f{2}{5}\wh{\eu}_{2}-\f{1}{5}\wh{\eu}_{3}+\f{6}{5}\wh{\ev}_{1}-\f{7}{5}\wh{\ev}_{2}+\f{1}{5}\wh{\ev}_{3}+\f{12}{5}\tau_1(\wh{\eu}_{1}{\color{black}+\alpha_0\wh{\ev}_{1}}),\\
\wt s \wh{\eu}_{2} &= \f{2}{13}\wh{\eu}_{1} - \f{5}{13}\wh{\eu}_{2}+\f{4}{13}\wh{\eu}_{3}-\f{1}{13}\wh{\eu}_{4}+\f{7}{13}\wh{\ev}_{1}-\f{8}{13}\wh{\ev}_{3}+\f{1}{13}\wh{\ev}_{4},\\
\wt s \wh{\ev}_{1} &= \f{6}{5}\wh{\eu}_{1}-\f{7}{5}\wh{\eu}_{2}+\f{1}{5}\wh{\eu}_{3}-\f{1}{5}\wh{\ev}_{1}+\f{2}{5}\wh{\ev}_{2}-\f{1}{5}\wh{\ev}_{3}+\f{12}{5}\tau_2(\wh{\eu}_{1}{\color{black}+\alpha_0\wh{\ev}_{1}}),\\
\wt s \wh{\ev}_{2} &= \f{7}{13}\wh{\eu}_{1}-\f{8}{13}\wh{\eu}_{3}+\f{1}{13}\wh{\eu}_{4}+\f{2}{13}\wh{\ev}_{1}-\f{5}{13}\wh{\ev}_{2}+\f{4}{13}\wh{\ev}_{3}-\f{1}{13}\wh{\ev}_{4}.
\end{align*}
Using \eqref{es5}-\eqref{es7}, we rewrite the above equations as a $4\times 4$ linear system 
\begin{equation}\label{BoundarySystem}
C_{s}(\wt s)\bm{\Sigma_s} = \m{d_s},    
\end{equation}
where $\bm{\Sigma_s}=[\sigma_0, \sigma_1, \sigma_2, \wh{\eu}_{1}]^\tra$ is the unknown vector, and the right-hand side 
\begin{align*}
\m{d_s}=h^2 &\left[-\f{1}{10}(3\wh v_{xx}(x_1)+2\wh u_{xx}(x_1)),\right.\\
&\f{1}{26}(3\wh v_{xx}(x_1)+2\wh u_{xx}(x_1)),\\
&-\f{1}{10}(3\wh u_{xx}(x_1)+2\wh v_{xx}(x_1)),\\
&\left.\f{1}{26}(3\wh u_{xx}(x_1)+2\wh v_{xx}(x_1))\right]
\end{align*}
is the truncation error vector. 

{\color{black}At $\wt s=0$, the boundary system matrix $C_s$ takes the form 
\begin{equation*}
\renewcommand{\arraystretch}{1.5}
C_s(0) = \begin{bmatrix}
\frac{\sqrt{33}+4}{10}, & \frac{96\alpha_0\tau_1-5\sqrt{33}+73}{40}, & \frac{24\alpha_0\tau_1+7}{10}, & -\frac{12\tau_1(\alpha_0-1)+7}{5} \\
-\frac{\sqrt{33}+4}{26}, & \frac{5\sqrt{33}+23}{104}, & \frac{5}{26}, & -\frac{5}{13} \\
-\frac{\sqrt{33}+4}{10}, & \frac{96\alpha_0\tau_2-5\sqrt{33}+17}{40}, & \frac{24\alpha_0\tau_2-7}{10}, & \frac{7-12\tau_2(\alpha_0-1)}{5}\\
\frac{\sqrt{33}+4}{26}, &\frac{5\sqrt{33}-17}{104}, & -\frac{5}{26}, &\frac{5}{13}
\end{bmatrix},
\renewcommand{\arraystretch}{1.0}
\end{equation*}
which depends on the parameter $\alpha_0$ from the boundary condition, and the penalty parameters $\tau_1,\tau_2$. We solve equation \eqref{BoundarySystem} at $\tilde{s}=0$ and focus on the third component of the solution, $\sigma_2$, that is multiplied with the slowly decaying $\thet_2=1-\mathcal{O}(\wt s)$ in the general solution. We have 
\[
\sigma_2 = -\frac{(25\sqrt{33}-15)(\alpha_0\tau_1+\tau_2+1)}{204(\tau_1+\tau_2)(\alpha_0+1)},
\]
which is well-defined because $\alpha_0\geq 0$ so that $\alpha_0+1\neq 0$, and $\tau_1+\tau_2=0$ violates the stability condition \eqref{StabilityCondition}. }

It is important to note that $\sigma_2=0$ when $\alpha_0\tau_1+\tau_2+1=0$, regardless of the relation between $\wh v_{xx}(x_1)$ and $\wh u_{xx}(x_1)$ in the truncation error vector $\m{d_s}$. This is important, because $\sigma_2$ is multiplied with the slowly decaying component $\thet_2=1-\mathcal{O}(\wt s)$. Componentwise, we can estimate $\bm{\Sigma_s}$ as
\begin{align*}
 |\sigma_0|&\leq C h^2 (|\wh u_{xx}(x_1)|+|\wh v_{xx}(x_1)|), \\
 |\sigma_1|&\leq C h^2 (|\wh u_{xx}(x_1)|+|\wh v_{xx}(x_1)|), \\
 \sigma_2&=0,\\
 \wh{\eu}_{1} &\leq C h^2 (|\wh u_{xx}(x_1)|+|\wh v_{xx}(x_1)|).
\end{align*}

The error $\bm{\ve}$ in Laplace space can be estimated  as 
\begin{align*}
\|\wh{\bm{\ve}}\|_h &= \sqrt{h \sum_{j=1}^\infty |\wh\eu_j|^2+h \sum_{j=1}^\infty |\wh\ev_j|^2} \\
&=h^{0.5}\sqrt{|\wh\eu_1|^2+\f{1}{4}\sum_{j=2}^\infty |\sigma_0\kp^{j-2}+\sigma_1\thet_1^{j-1}|^2+|\wh\ev_1|^2+\f{1}{4}\sum_{j=2}^\infty |\sigma_0\kp^{j-2}-\sigma_1\thet_1^{j-1}|^2 }\\
&\leq h^{0.5} \sqrt{|\wh\eu_1|^2+|\wh\ev_1|^2+ \f{1}{2}\f{|\sigma_0|^2}{1-|\kp|^2} +\f{1}{2}\f{|\sigma_1|^2}{1-|\thet_1|^2}}\\
&\leq Ch^{2.5} (|\wh u_{xx}(x_1)|+|\wh v_{xx}(x_1)|).
\end{align*}
The other case for the right boundary $x=1$ can be analyzed in the same way. We can thus conclude that the overall convergence rate is 2.5 when \eqref{AlphaTau} is satisfied. In this case, the convergence rate is 1.5 orders higher than the boundary truncation error of the upwind SBP operators. 

If \eqref{AlphaTau} is not satisfied, then all components of $\bm{\Sigma_s}$ are of the same order as the truncation error, and a convergence rate of 2 follows, which is one order higher than the boundary truncation error of the upwind SBP operator. 
\end{proof}


In the numerical experiments in Sec.~\ref{sec_ne}, we also verify that the same \textit{gain in convergence} follows for higher order accurate schemes with $p=4$, 5, 6, 7, 8. The procedure of deriving the corresponding error estimates is the same as above for $p=3$.

\section{An upwind SBP-SAT discretization with the WENO property} \label{sec-SBPWENO}
In this section, we construct SBP-SAT discretizations with the WENO property for the model problem \eqref{eqn_adv}. 

To effectively compute weak solutions, which may contain discontinuities, we would use conservative difference to approximate the derivative $u_x$ at the grid point $x_j$.
To this end, we complement the grid $\m{x}$ in \eqref{1d_grid} by $\bar{\m{x}}:=[\bar x_0, \bar x_1,\cdots,\bar x_n]$, which contains $n+1$ flux points with $\bar h_i = \bar x_i-\bar  x_{i-1}$. 
In particular, the first and last flux points coincide with the domain boundaries, i.e., $\bar x_0 = x_1$ and $\bar x_n = x_n$. In the interior, the flux point $\bar x_i$ is the midpoint of $x_i$ and $x_{i+1}$. However, a few flux points close to the boundaries are shifted, resulting in a nonuniform grid of $\bar{\m{x}}$. The exact location of those flux points will be specified in the derivation of the scheme. 
Next, we write the discretization matrix in the conservative form,
\begin{equation}\label{Dm_flux}
D_{m}\m{u}|_i= \frac{\hat u_i-\hat u_{i-1}}{\bar h_i}, 
\quad i=1, \ldots, n,
\end{equation}
where the hat variable $\hat u_i$ is the numerical flux,
which typically is a Lipschitz continuous function of several neighboring values $u_i$.

The essence of WENO is to 
use a nonlinear convex combination of numerical fluxes from all the candidate small stencils,
such that the scheme can achieve arbitrarily high order accuracy in smooth regions and resolve shocks or other discontinuities sharply and in an essentially non-oscillatory fashion.

In the following, we consider the cases $p=3$ and $p=4$ in Sec.~\ref{sec-p3} and Sec.~\ref{sec-p4}, respectively, followed by the stabilization technique in Sec.~\ref{sec-stabilization}.

\subsection{Interior order $p=3$}\label{sec-p3}

In the interior, the upwind SBP operator $D_m$ with third order accuracy has a four-point stencil that is biased to the left,
\begin{equation}\label{p3_int}
D_m \m{u}|_i = \frac{1}{h}\left(\frac{1}{6}u_{i-2}-u_{i-1}+\frac{1}{2}u_i+\frac{1}{3}u_{i+1}\right), \quad i = 3,4,\cdots, n-2.
\end{equation}
The boundary closures are 
\begin{align*}
&D_m \m{u}|_1 = \frac{1}{h}(-u_{1}+u_{2}),\quad D_m \m{u}|_2 = \frac{1}{h}\left(-\frac{9}{13}u_{1}+\frac{5}{13}u_{2}+\frac{4}{13}u_{2}\right), \\
&D_m \m{u}|_{n-1} = \frac{1}{h}\left(\frac{2}{13}u_{n-3}-\frac{12}{13}u_{n-2}+\frac{5}{13}u_{n-1}+\frac{5}{13}u_{n}\right),\\\
&D_m \m{u}|_{n} = \frac{1}{h}\left(\frac{2}{5}u_{n-2}-\frac{9}{5}u_{n-1}+\frac{7}{5}u_{n}\right).
\end{align*}
These boundary closures are designed so that the SBP property is satisfied, but are only first order accurate. The weights of the SBP norm on the first two grid points are $\frac{5}{12}$ and $\frac{13}{12}$. As a consequence \cite{Fisher2011}, we have 
\begin{align*}
&\bar x_1 =\bar x_0+ \frac{5}{12}h,\;\;\; \bar x_{n-1} =\bar x_n- \frac{5}{12}h, \\
&\bar x_i=\frac{1}{2}(x_i+x_{i+1}),\quad i = 2,3,\cdots, n-2.
\end{align*}

Next, we rewrite the difference stencils in $D_m$ in the form of numerical fluxes \eqref{Dm_flux}. 
For each interior flux point $2\leq i\leq n-2$, numerical fluxes \eqref{p3_int} can be represented in the form 
\begin{align*}
\hat u_i =& -\frac{1}{3} u_{i-1} +\frac{7}{6}u_{i} + \frac{1}{6} u_{i+1} 
=\frac{1}{3} \left(\frac{1}{2}u_i+\frac{1}{2}u_{i+1}\right)+\frac{2}{3} \left(-\frac{1}{2}u_{i-1}+\frac{3}{2}u_{i}\right) \\
:=& d_i^{(1)} \hat u_i^{(1)}+d_i^{(2)} \hat u_i^{(2)}
\end{align*}
Here, $\hat u_i^{(1)}$ and $\hat u_i^{(2)}$ are numerical fluxes on the two candidate stencils $\{x_{i}, x_{i+1}\}$ and $\{x_{i-1}, x_{i}\}$, respectively.  
These two linear weights $d_i^{(1)}$ and $d_i^{(2)}$ will be turned into nonlinear weights when adding the WENO property.

On the first flux point $\bar x_0$, we take the numerical flux to be $\hat u_0 = u_1$. By using the same requirement  \eqref{Dm_flux} and the boundary stencil of $D_m$ with $i=1$ and 2, we obtain the numerical flux $\hat u_1 = \frac{7}{12}u_1+\frac{5}{12}u_2$. Note that the numerical fluxes on the first two flux points have fixed stencils.  On the right boundary, we again take $\hat u_n=u_n$. For the numerical flux on $i = n-1$, we make the ansatz 
$$\hat u_{n-1}= d_{n-1}^{(1)} \left(-\frac{7}{12}u_{n-2}+\frac{19}{12}u_{n-1}\right)+d_{n-1}^{(2)} \left(\frac{5}{12}u_{n-1}+\frac{7}{12}u_{n}\right).$$ 
By requiring \eqref{Dm_flux} with $i=n-1$, we obtain $d_{n-1}^{(1)}=\frac{2}{7}$ and $d_{n-1}^{(2)}=\frac{5}{7}$, which will also be changed to nonlinear weights in the SBP-WENO scheme.

There exist many nonlinear weights with different properties. For example, in \cite{Yamaleev2009} the following nonlinear weights are derived, $i=2,3,\cdots,n-1$, which follows the idea of WENOZ,
\begin{align*}
&\tau_i = (u_{i+1}-2u_i+u_{i-1})^2,\\
&\beta_i^{(1)} = (u_i-u_{i-1})^2,\quad \beta_i^{(2)} = (u_{i+1}-u_{i})^2,\\
&w_i^{(j)} = \frac{\alpha_i^{(j)}}{\alpha_i^{(1)}+\alpha_i^{(2)}}, \quad 
\alpha_i^{(j)} = d_i^{(j)}\left(1+\frac{\tau_i}{\varepsilon+\beta_i^{(j)}}\right),\ j=1,2.
\end{align*}
Here, the small constant $\varepsilon$ avoids division by zero. 
Replacing the linear weights by the above nonlinear weights in the corresponding numerical fluxes, we obtain the difference operator  
\begin{equation}\label{Dmw_flux}
D_{mw}\m{u}|_i= \frac{\hat u^w_i-\hat u^w_{i-1}}{\bar h_i}, 
\end{equation}
with the nonlinear numerical fluxes 
$$\hat{u}^{w}_{i} =\omega_i^{(1)}\hat{u}_i^{(1)} + \omega_i^{(2)}\hat{u}_i^{(2)}.$$ 
The resulting difference operator $D_{mw}$ has the desired WENO property, however, the SBP property is lost. As a consequence, energy stability cannot be proved.

\subsection{Interior order $p=4$} \label{sec-p4}

In the interior, the upwind SBP operator $D_m$ with fourth order accuracy has a five-point stencil that is biased to the left,
\begin{equation}\label{p4_int}
D_m \m{u}|_i = \frac{1}{h}\left(-\frac{1}{12}u_{i-3}+\frac{1}{2} u_{i-2} - \frac{3}{2}u_{i-1}+\frac{5}{6}u_i+\frac{1}{4}u_{i+1}\right), \quad i = 5,6,\cdots, n-4,
\end{equation}
which can be rewritten in the conservative form with numerical fluxes, 
\begin{align*}
 \hat{u}_{i} =& \frac{1}{12}u_{i-2}-\frac{5}{12} u_{i-1} + \frac{13}{12}u_{i}+\frac{1}{4}u_{i+1} 
:= d_i^{(1)} \hat u_i^{(1)}+d_i^{(2)} \hat u_i^{(2)} +d_i^{(2)} \hat u_i^{(2)}. 
\end{align*}
Here, the linear weights are 
\begin{align}\label{d123}
d_i^{(1)} = \frac{1}{2}, \quad  d_i^{(2)}=\frac{1}{4}, \quad  d_i^{(3)}=\frac{1}{4},
\end{align}
and the numerical fluxes on substencils are 
\begin{align*}
    \hat{u}_i^{(1)}= \frac{1}{2}u_i+\frac{1}{2}u_{i+1}, \quad
\hat{u}_i^{(2)}=-\frac{1}{2}u_{i-1}+\frac{3}{2}u_{i},\quad
\hat{u}_i^{(3)}=\frac{1}{3}u_{i-2}-\frac{7}{6}u_{i-1}+\frac{11}{6}u_{i}.
\end{align*}
The flux points are $\bar x_i=\frac{1}{2}(x_i+x_{i+1})$ in the interior $i = 4,5,\cdots, n-4$, and are shifted near the boundaries,
\begin{align*}
&\bar{x}_0 = x_1, \quad \bar x_1 =\bar x_0+ \frac{49}{144}h, \quad
\bar x_2 = \bar x_1 + \frac{61}{48}h, \quad \bar x_3 =\bar x_2 + \frac{41}{48}h, \\
& \bar x_n = x_n, \quad \bar x_{n-1} =\bar x_n- \frac{5}{12}h, \quad
\bar x_{n-2} = \bar{x}_{n-1} -\frac{61}{48}h, \quad \bar{x}_{n-3} = \bar x_{n-2}-\frac{41}{48} h.
\end{align*}

The numerical fluxes on each point $\bar x_i$ are constructed such that $D_m \m{u}|_i=\frac{\hat{u}_i - \hat{u}_{i-1}}{\bar h_i}$. WENO methodology can be applied by replacing the linear weights \eqref{d123} by nonlinear weights. More precisely, we have designed the following nonlinear weights with the help of smoothness indicator. In the interior, we have 
\begin{align*}
&\tau_i = (u_{i+1}-3u_i+3u_{i-1}-u_{i-2})^2,\\
&\beta_i^{(1)} = (u_{i+1}-u_{i})^2,\qquad \beta_i^{(2)} = (u_{i}-u_{i-1})^2,\\
& \beta_i^{(3)} = \frac{13}{12}(u_i-2u_{i-1}+u_{i-2})^2 + \frac{1}{4}(5u_i-8u_{i-1}+3u_{i-2})^2, \\
&w_i^{(j)} = \frac{\alpha_i^{(j)}}{\alpha_i^{(1)}+\alpha_i^{(2)}+\alpha_i^{(3)}}, \quad \alpha_i^{(j)} = d_i^{(j)}\left(1+\frac{\tau_i}{\varepsilon+\beta_i^{(j)}}\right),\ j=1,2,3.
\end{align*}
Note that when $u$ is smooth in $\{x_{i-2},\ldots, x_{i+1}\}$, we have 
\begin{align*}
    & \beta_1 = (u'_i)^2 h^2 + u'_i u''_i h^3 + \left( \frac{1}{4} (u''_i)^2 + \frac{1}{3}u'_i u^{(3)}_i\right)h^4 + \mathcal{O}(h^5),\\
    & \beta_2 = (u'_i)^2 h^2 - u'_i u''_i h^3 + \left( \frac{1}{4} (u''_i)^2 + \frac{1}{3}u'_i u^{(3)}_i\right)h^4 + \mathcal{O}(h^5),\\
    & \beta_3 = (u'_i)^2 h^2 + 2u'_i u''_i h^3 + \left( \frac{25}{12} (u''_i)^2 - \frac{32}{12}u'_i u^{(3)}_i\right)h^4 + \mathcal{O}(h^5),\\
    & \tau = (u^{(3)}_i)^2 h^6 + \mathcal{O}(h^7).
\end{align*}
Suppose $\varepsilon=0$, the above equations lead to
\begin{align*}
    \alpha^{(j)}_i = d^{(j)}_i \left( 1 + \mathcal{O}(h^4)\right), \;\; j=1, 2, 3.
\end{align*}
Furthermore, $w_i^{(j)}=d_i^{(j)}(1+\mathcal{O}(h^4))$. Consequently, 
\begin{align*}
    \sum_{j=1}^{3} \left(w_i^{(j)} \hat u_i^{(j)}  - d_i^{(j)} \hat u_i^{(j)} \right)
    = \mathcal{O}(h^4).
\end{align*}

This indicates that the WENO methodology does not destroy the fourth order accuracy in the interior. Moreover, to reduce the impact of $\epsilon$, we take $\epsilon=h^2$ in our numerical simulation. Thus, $\epsilon$ has the same order of magnitude as those smoothness indicators in smooth cases and the above analysis is still valid.

In the Appendix, we present the formulas of the WENO schemes near the boundaries.
{It is noticed that the number and size of small stencils vary with the location $\bar{x}_i$. 
Essentially, we choose the small stencil to meet the conditions that numerical flux on each substencil attains low order accuracy and all linear weights are positive.

\subsection{Stabilization}\label{sec-stabilization}
As discussed, the $D_{mw}$ does not have the SBP property. To derive an SBP-WENO scheme, we add a stabilization term to $D_{mw}$ such that the resulting operator, $D_{mws}$, has both the SBP and WENO property. We follow the main steps from \cite{Yamaleev2009,Yamaleev2009b,Fisher2011}, and 
 decompose $D_{mw}$ to $D_{mw}=H^{-1}(Q_{mw}+R_{mw})$, where the skew-symmetric part is $Q_{mw}=\frac{1}{2}(HD_{mw}-(HD_{mw})^\tra+B)$ and the symmetric part is $R_{mw}=\frac{1}{2}(HD_{mw}+(HD_{mw})^\tra-B)$. Because of the  dependence of $R_{mw}$ on the nonlinear weights, the matrix $R_{mw}$ is not guaranteed to be positive semidefinite. To this end, we make the ansatz $D_{mws}:=H^{-1}(Q_{mw}+R_{mw}+R_s)$, where the stabilization term $R_s$ leads to a positive semidefinite matrix  $R_{mw}+R_s$ so that $D_{mws}$ is an SBP operator with the desired accuracy property. 
 
 To construct $R_s$, we decompose $R_{mw}$. As an example of $p=4$, we have  $$R_{mw}=\Delta\Delta^\tra\Delta\Lambda_3\Delta^\tra\Delta\Delta^\tra+\Delta\Delta^\tra\Lambda_2\Delta\Delta^\tra+\Delta\Lambda_1\Delta^\tra,$$
 where $\Delta$ is an $n$-by-$n+1$ matrix with only nonzero elements $\Delta_{i,i}=-1$ and $\Delta_{i,i+1}=1$ for $i=1,2,\cdots,n$. The diagonal matrices $\Lambda_1$, $\Lambda_2$ and $\Lambda_3$ are in general not positive semidefinite. Using them, we construct new diagonal matrices $\Lambda_{1s}$, $\Lambda_{2s}$ and $\Lambda_{3s}$ with diagonal elements $(\Lambda_{1s})_{ii}=\frac{1}{2}\left(\sqrt{((\Lambda_{1})_{ii})^2+(\delta_1)^2}-(\Lambda_{1})_{ii}\right)$, $(\Lambda_{2s})_{ii}=\frac{1}{2}\left(\sqrt{((\Lambda_{2})_{ii})^2+(\delta_2)^2}-(\Lambda_{2})_{ii}\right)$ and $(\Lambda_{3s})_{ii}=\frac{1}{2}\left(\sqrt{((\Lambda_{3})_{ii})^2+(\delta_3)^2}-(\Lambda_{3})_{ii}\right)$ for some constant $\delta_1$, $\delta_2$, and $\delta_3$ that depend on $h$. 
 In particular, for the inner point $i=5,6,\cdots,n-4$, 
 \begin{equation*}
    \begin{aligned}
        (\Lambda_{1})_{ii}=& \frac{1}{4}w^{(2)}_{i-1} 
        +\frac{1}{4}w^{(3)}_{i-1} -\frac{1}{4}w^{(2)}_{i} 
        -\frac{5}{12}w^{(3)}_{i} +\frac{1}{6}w^{(3)}_{i+1},\\
        (\Lambda_{2})_{ii}=& \frac{1}{4}w^{(2)}_{i}
        +\frac{1}{12}w^{(3)}_{i} - \frac{1}{3}w^{(3)}_{i+1},\\
        (\Lambda_{3})_{ii}=& \frac{1}{6}w^{(3)}_{i}. 
    \end{aligned}
 \end{equation*}
 It is observed that $(\Lambda_{3})_{ii}$ is always non-negative and we can set $\Lambda_{3,s}=\Lambda_{3}$ directly, but modification on $(\Lambda_{1})_{ii}$ and $(\Lambda_{2})_{ii}$ is necessary. Moreover, for smooth cases, $(\Lambda_{1})_{ii}=\mathcal{O}(h^4)$ and $(\Lambda_{2})_{ii}=\mathcal{O}(h^4)$. Therefore, we can take $\delta_1=\delta_2$ to be $h^4$ to avoid affecting accuracy. 
 Finally, let $R_{s}=\Delta\Delta^\tra\Delta\Lambda_{3s}\Delta^\tra\Delta\Delta^\tra+\Delta\Delta^\tra\Lambda_{2s}\Delta\Delta^\tra+\Delta\Lambda_{1s}\Delta^\tra$, which leads to the positive semidefiniteness of $R_{mw}+R_s$.

 The case with $p=3$ is very similar to the case $p=4$, with $\Lambda_3$ being a zero matrix. To avoid repetition, details are not described in this article.

\section{Numerical experiments}\label{sec_ne}
In this section, we present numerical examples to verify the theoretical analysis. We start with convergence studies for problems with smooth solutions, and consider the linear SBP-SAT discretization in Section \ref{ex_linear_scheme}, and the WENO scheme in Section \ref{sec_exp_WENO}. In addition, we also use the SBP-WENO method to solve a problem with nonsmooth solutions. 

\subsection{Convergence rate for the linear schemes}\label{ex_linear_scheme}
We consider the advection equation \eqref{eqn_adv} with a manufactured smooth solution $U(x,t) = \sin(2\pi(x-t)+1)$. 
{\color{black}We discretize in space by the upwind operators with interior accuracy $q=3$, 4, 5, 6, 7, and 8.} We choose a stepsize in the Runge-Kutta time integrator small enough so that the error is dominated by the spatial discretization. 

The discretization is stable for any penalty parameter $\tau\leq -1/2$. According to the accuracy analysis in Section \ref{sec_err_adv} for $p=3$, the convergence rate is 2.5 when $\tau=-1$, and is 2 when $\tau\neq -1$. That is, the convergence rate is 1.5 orders higher than the boundary truncation error when $\tau=-1$, and one order higher for other choices of $\tau$. This is clearly observed in the first error plot in Figure \ref{Ex1_order3}. 
{\color{black}In addition, we also observe the same phenomena for higher order discretizations $q=4$, 5, 6, 7 and 8, see Figure \ref{Ex1_order3}}. This indicates that $\tau=-1$ is a good choice even for higher order accurate discretizations. 

\begin{figure}
\centering
\begin{subfigure}{0.49\textwidth}
\includegraphics[width=\textwidth]{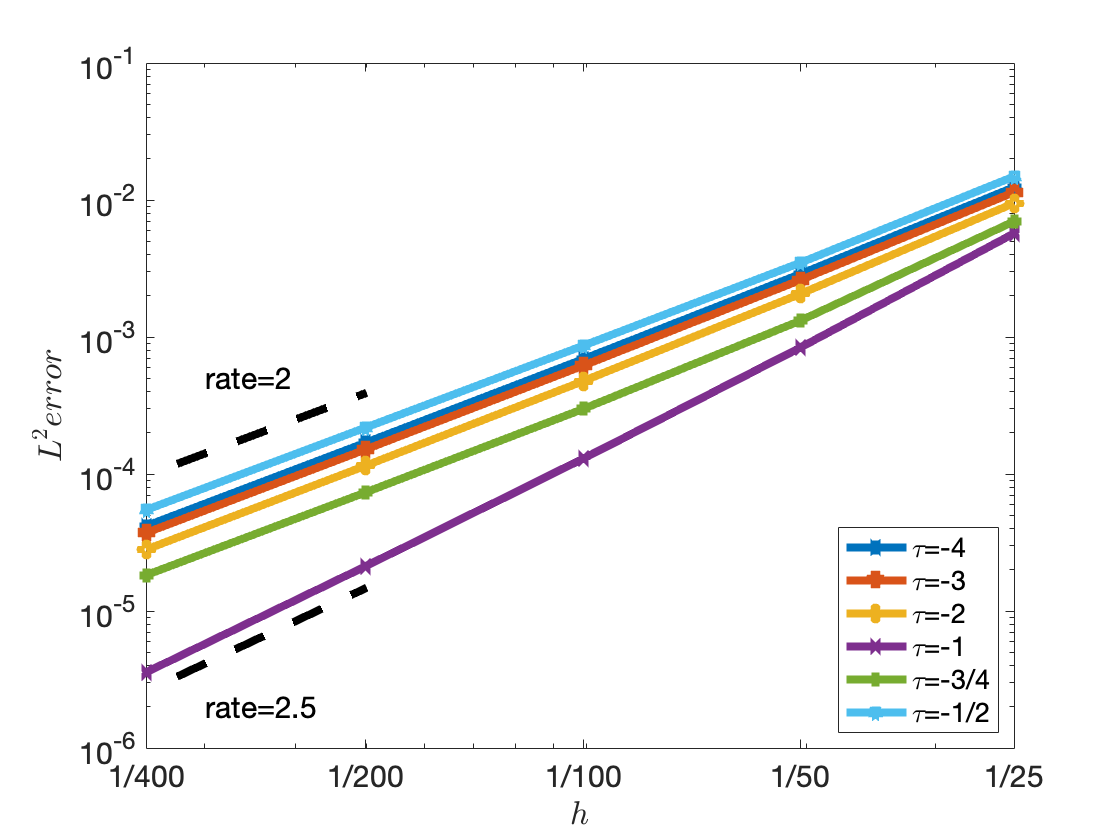}
\caption{$p=3$.}
\end{subfigure}
\begin{subfigure}{0.49\textwidth}
\includegraphics[width=\textwidth]{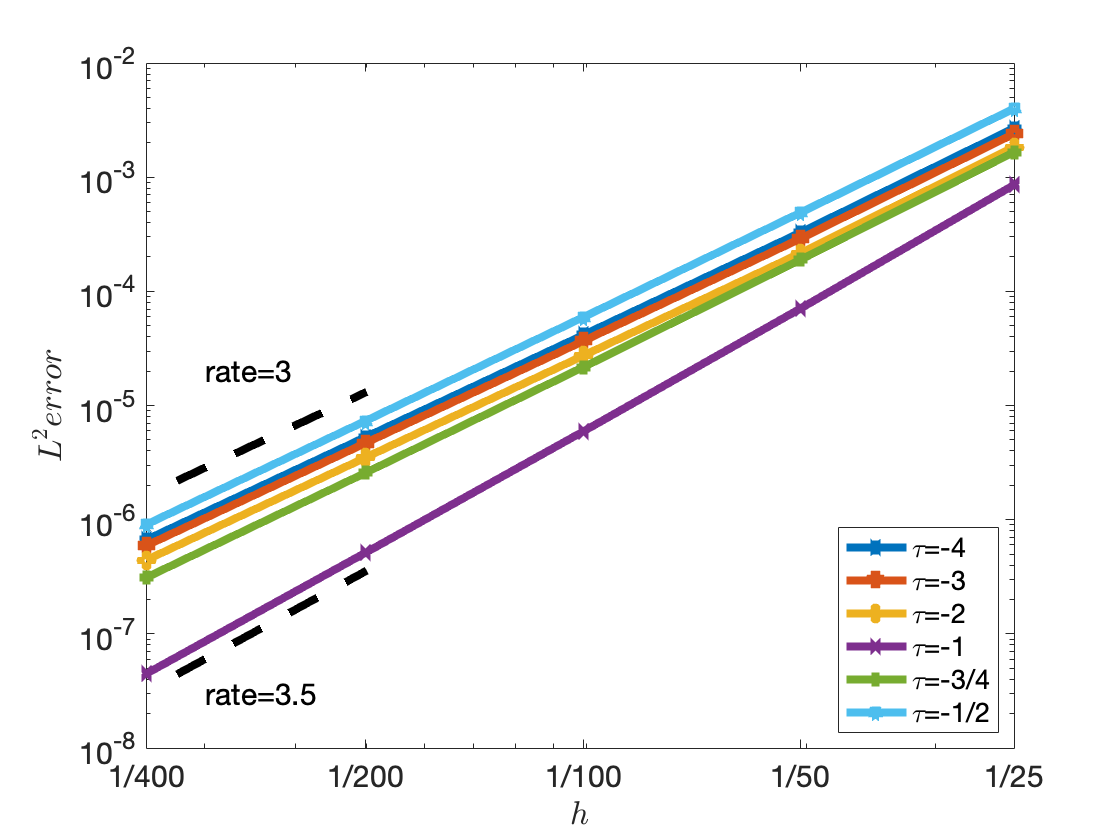}
\caption{$p=4$.}
\end{subfigure}\\
\begin{subfigure}{0.49\textwidth}
\includegraphics[width=\textwidth]{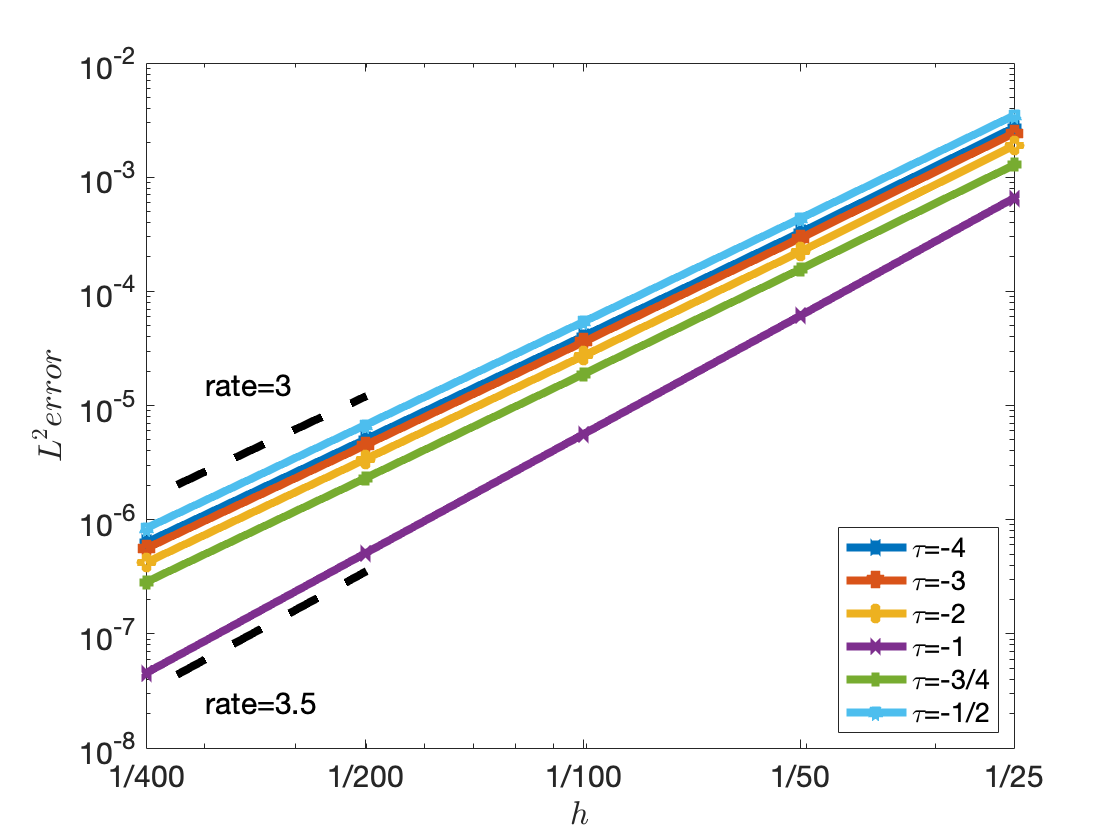}
\caption{\color{black}$p=5$.}
\end{subfigure}
\begin{subfigure}{0.49\textwidth}
\includegraphics[width=\textwidth]{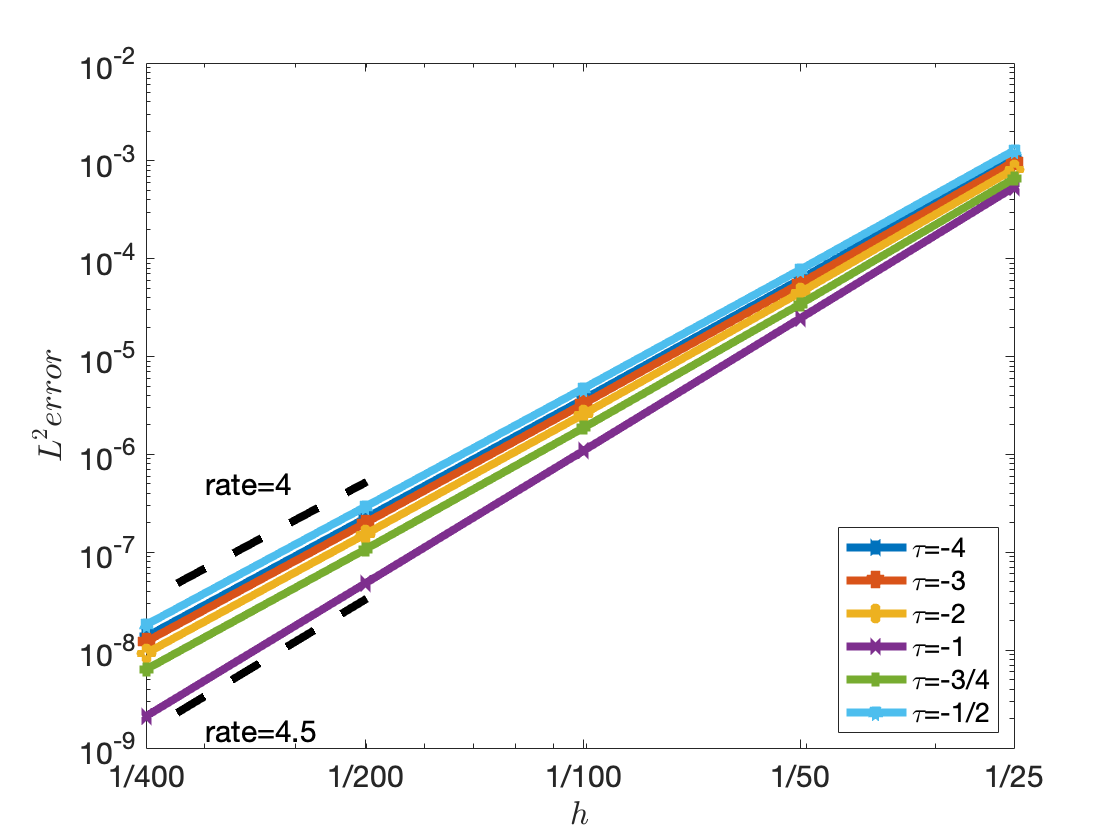}
\caption{\color{black}$p=6$.}
\end{subfigure}\\
\begin{subfigure}{0.49\textwidth}
\includegraphics[width=\textwidth]{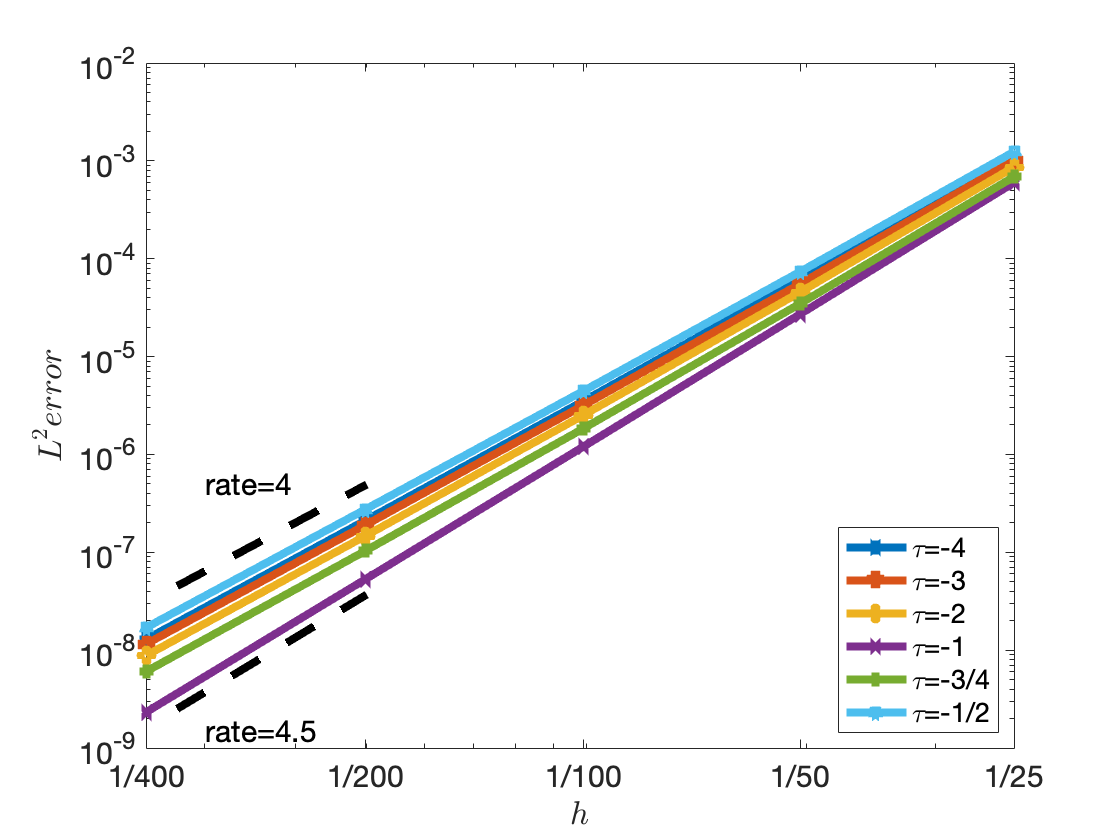}
\caption{\color{black}$p=7$.}
\end{subfigure}
\begin{subfigure}{0.49\textwidth}
\includegraphics[width=\textwidth]{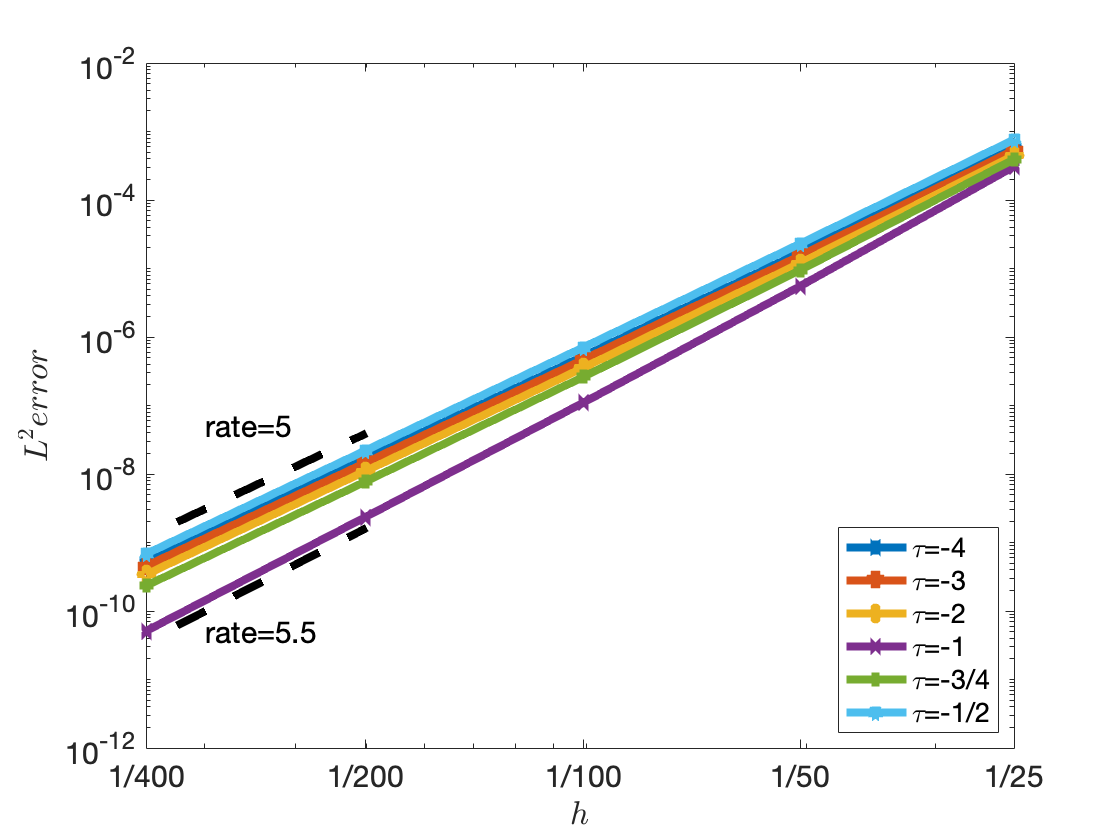}    
\caption{\color{black}$p=8$.}
\end{subfigure}
\caption{$L^2$ error for solving the advection equation \eqref{eqn_adv} with different values of $p$ and $\tau$.}
\label{Ex1_order3}
\end{figure}

{\color{black}Next, we consider the first order hyperbolic system \eqref{hs} with a manufactured smooth solution
\[
U = \begin{bmatrix}
-\sin(2\pi(x+t))+\cos(2\pi(x-t))\\
        \sin(2\pi(x+t))+\cos(2\pi(x-t))
        \end{bmatrix},
\]
and final time $T=1$. In the boundary condition \eqref{hs_bc}, we choose $\alpha_0=1/2$ and  $\alpha_1=0$. We choose the penalty parameters $\tau_2=-1/3$, $\tau_3=0$, $\tau_4=1$, and vary $\tau_1$. According to the error estimate in Theorem 6, the convergence rate is 2.5 if $\alpha_0\tau_1+\tau_2+1=0$, i.e. $\tau_1=-4/3$, while other stable choices of $\tau_1$ lead to a convergence rate of 2. This is clearly observed in the left panel of Figure \ref{Ex1_system}. 
In addition, we have also used upwind SBP operators with $p=4$, 5, 6, 7 and 8. 
When $\alpha_0\tau_1+\tau_2+1=0$, i.e. $\tau_1=-4/3$, the convergence rate is 1.5 orders higher than the boundary truncation error; other choices of $\tau_1$ gives a convergence rate that is one order higher than the boundary truncation error. This is consistent with the accuracy analysis for $p=3$.} 

\begin{figure}
\centering
\begin{subfigure}{0.49\textwidth}
\includegraphics[width=\textwidth]{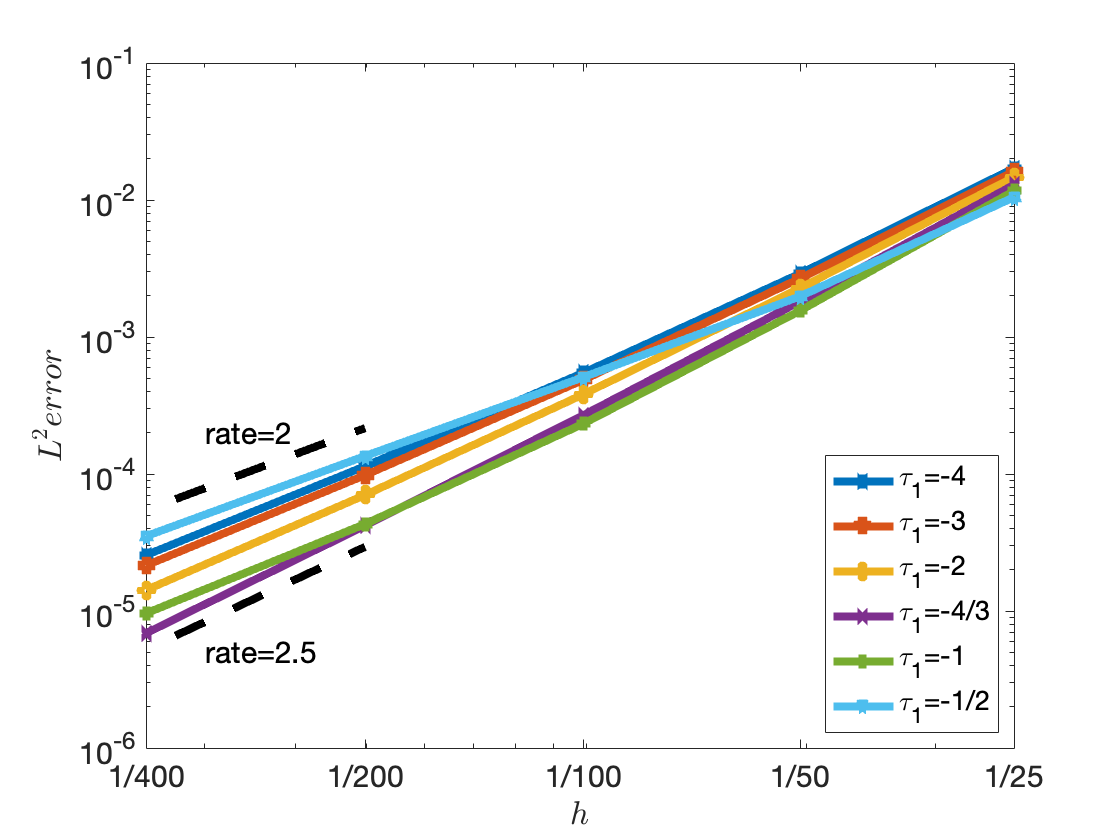}
\caption{$p=3$.}
\end{subfigure}
\begin{subfigure}{0.49\textwidth}
\includegraphics[width=\textwidth]{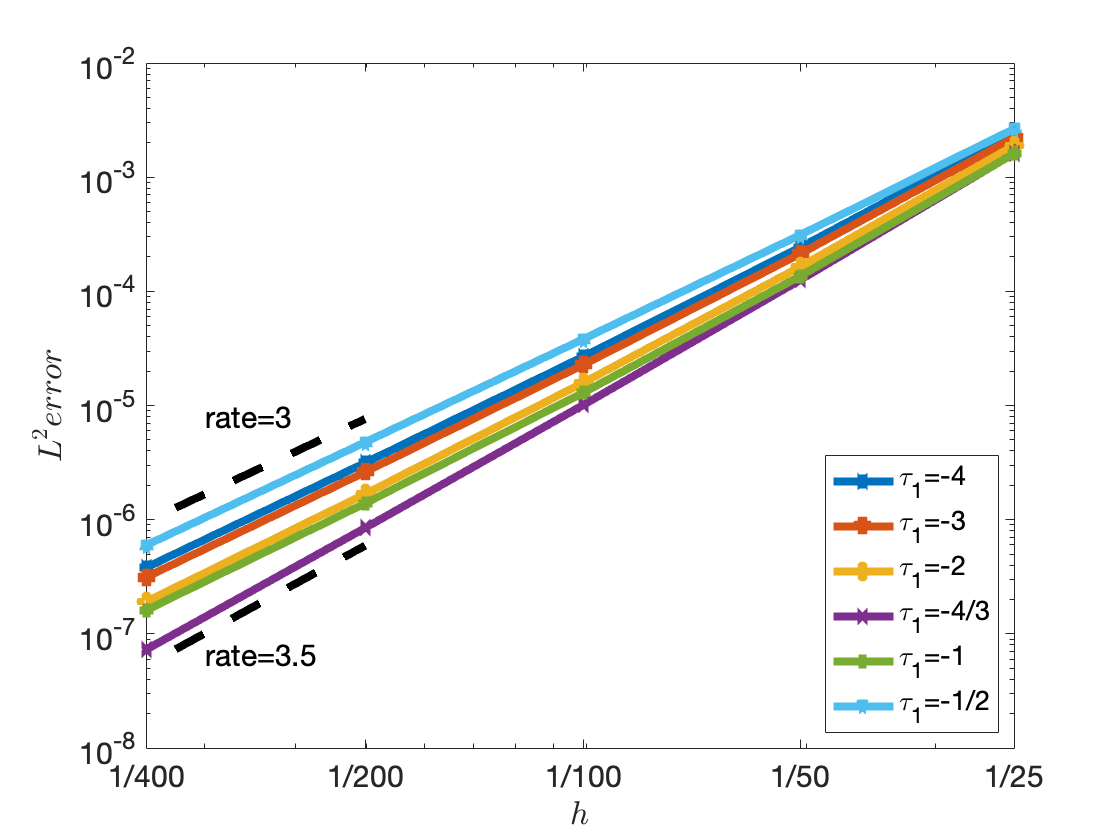}
\caption{$p=4$.}
\end{subfigure}\\
\begin{subfigure}{0.49\textwidth}
\includegraphics[width=\textwidth]{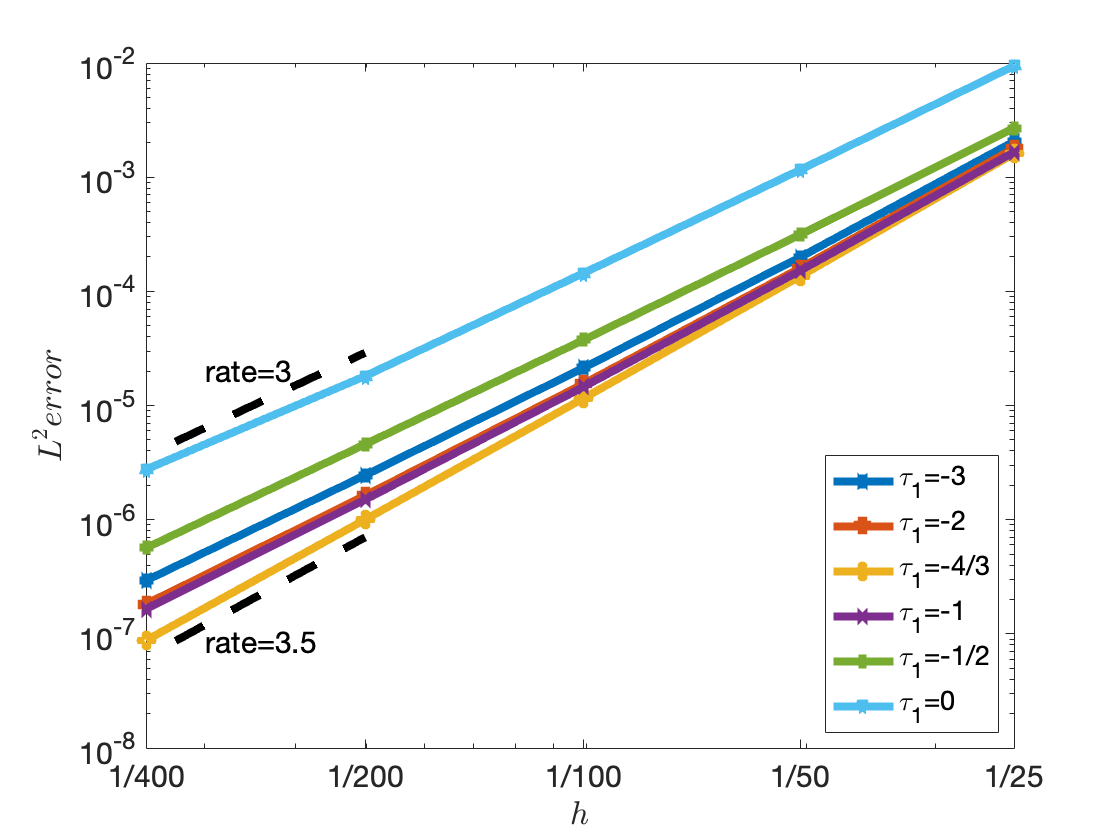}
\caption{\color{black}$p=5$.}
\end{subfigure}
\begin{subfigure}{0.49\textwidth}
\includegraphics[width=\textwidth]{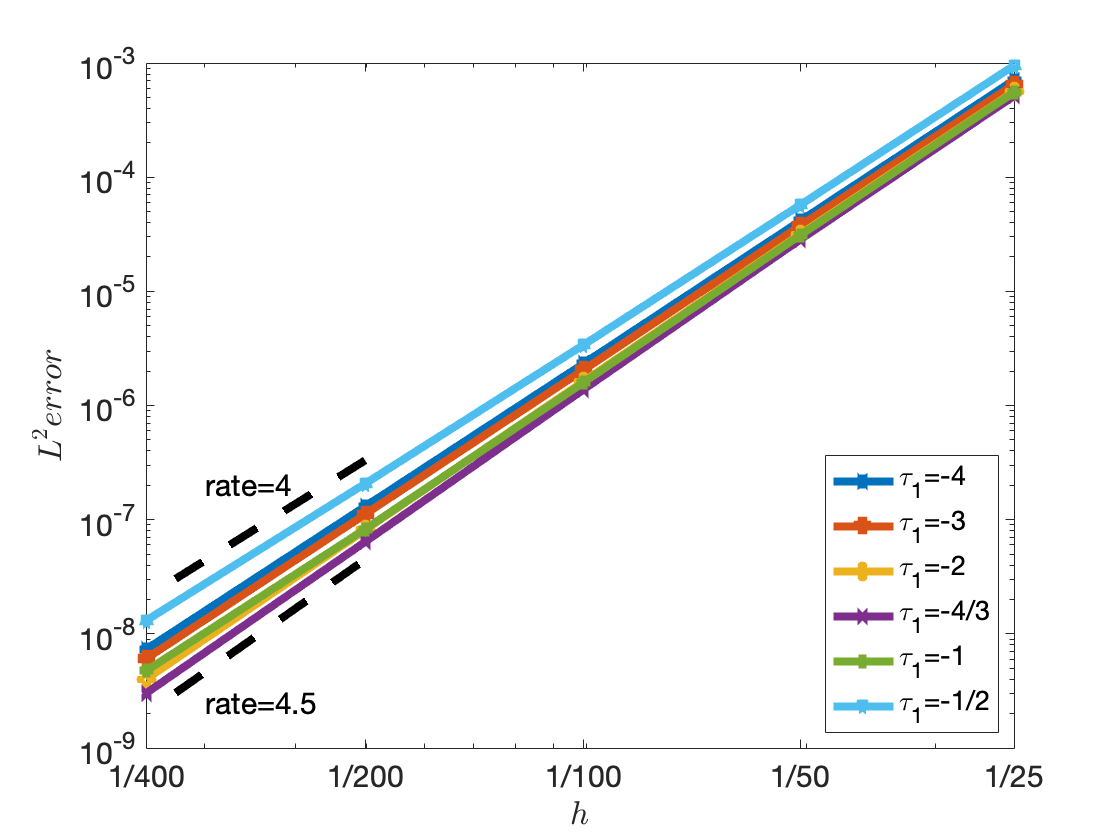}
\caption{\color{black}$p=6$.}
\end{subfigure}\\
\begin{subfigure}{0.49\textwidth}
\includegraphics[width=\textwidth]{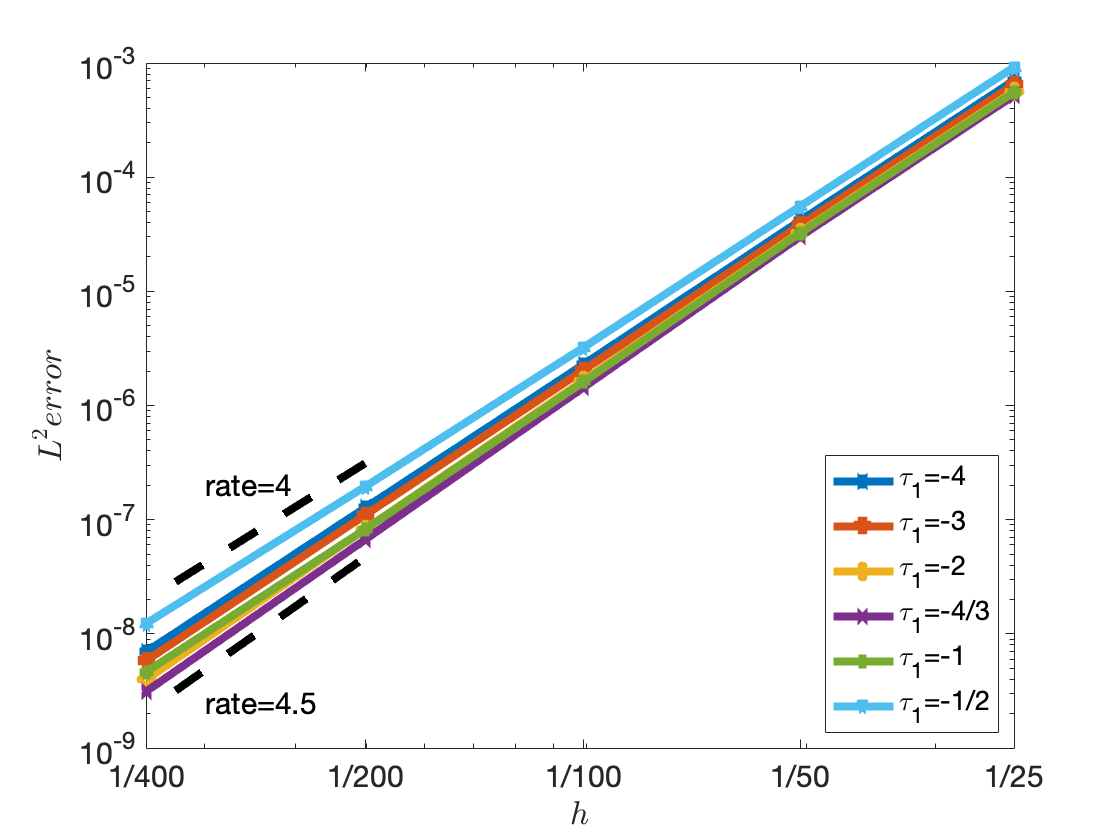}
\caption{\color{black}$p=7$.}
\end{subfigure}
\begin{subfigure}{0.49\textwidth}
\includegraphics[width=\textwidth]{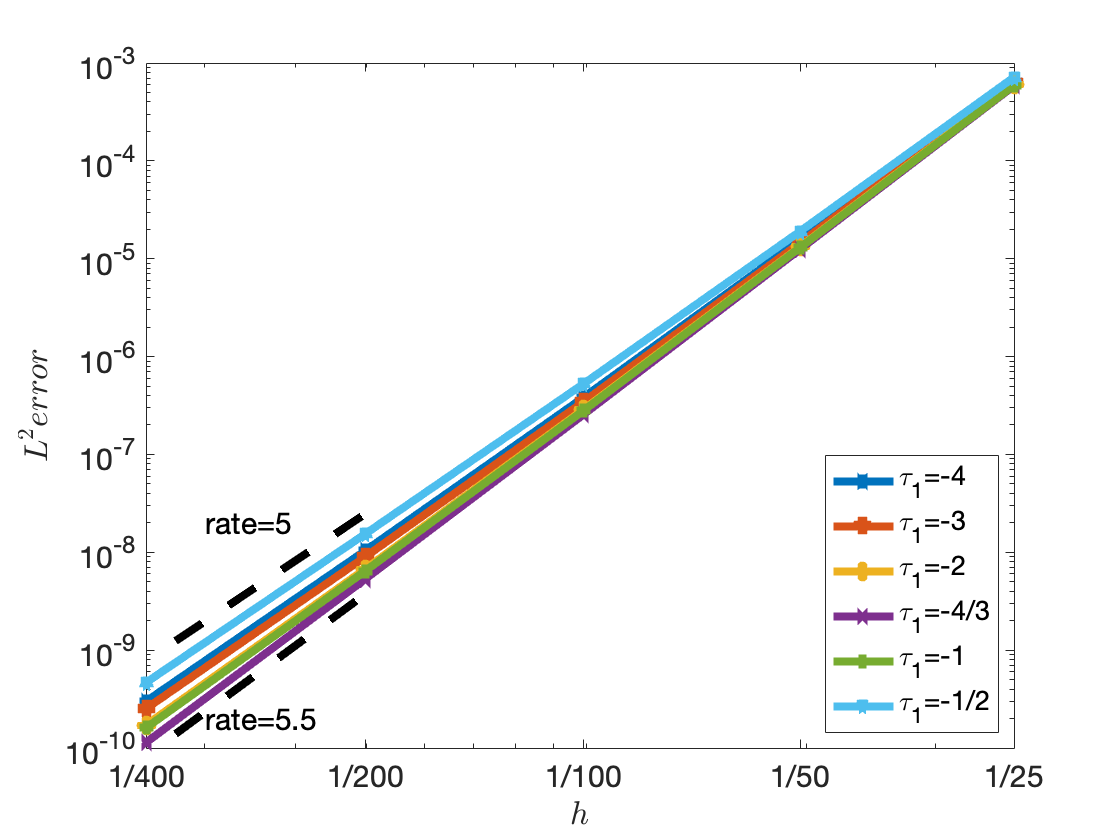}
\caption{\color{black}$p=8$.}
\end{subfigure}
\caption{$L^2$ error solving the hyperbolic system \eqref{hs} with different values of $p$ and $\tau_1$ for the hyperbolic system.}
\label{Ex1_system}
\end{figure}

\subsection{SBP-WENO}\label{sec_exp_WENO}
We now consider the SBP-WENO method developed in Section \ref{sec-SBPWENO}. In all experiments, we choose the parameters $\epsilon= h^2$ and $\delta=h^4$.

First, we consider the same smooth problem as in Sec.~\ref{ex_linear_scheme} for the advection equation. The error is plotted in Figure \ref{Ex2_order3_smooth}. Clearly, the WENO stencils and the stabilization terms do not destroy the accuracy property of the original SBP discretization.

\begin{figure}
\centering
\includegraphics[width=0.49\textwidth]{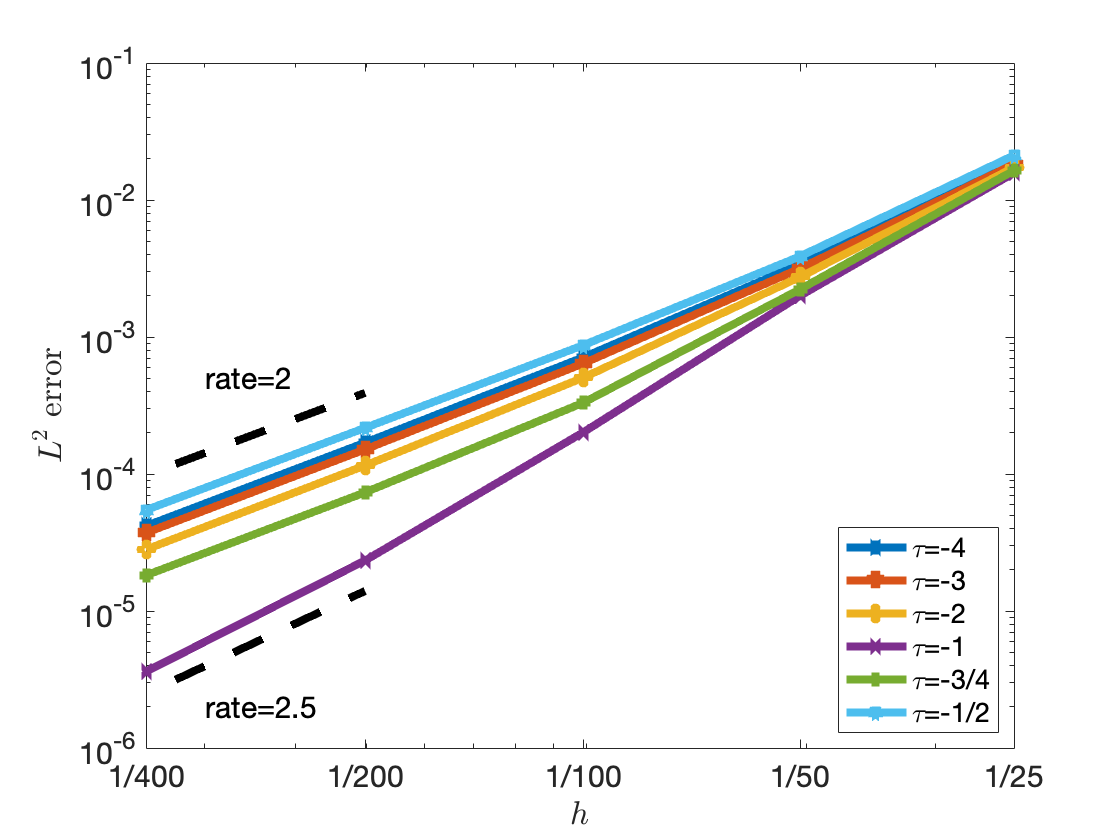}
\includegraphics[width=0.49\textwidth]{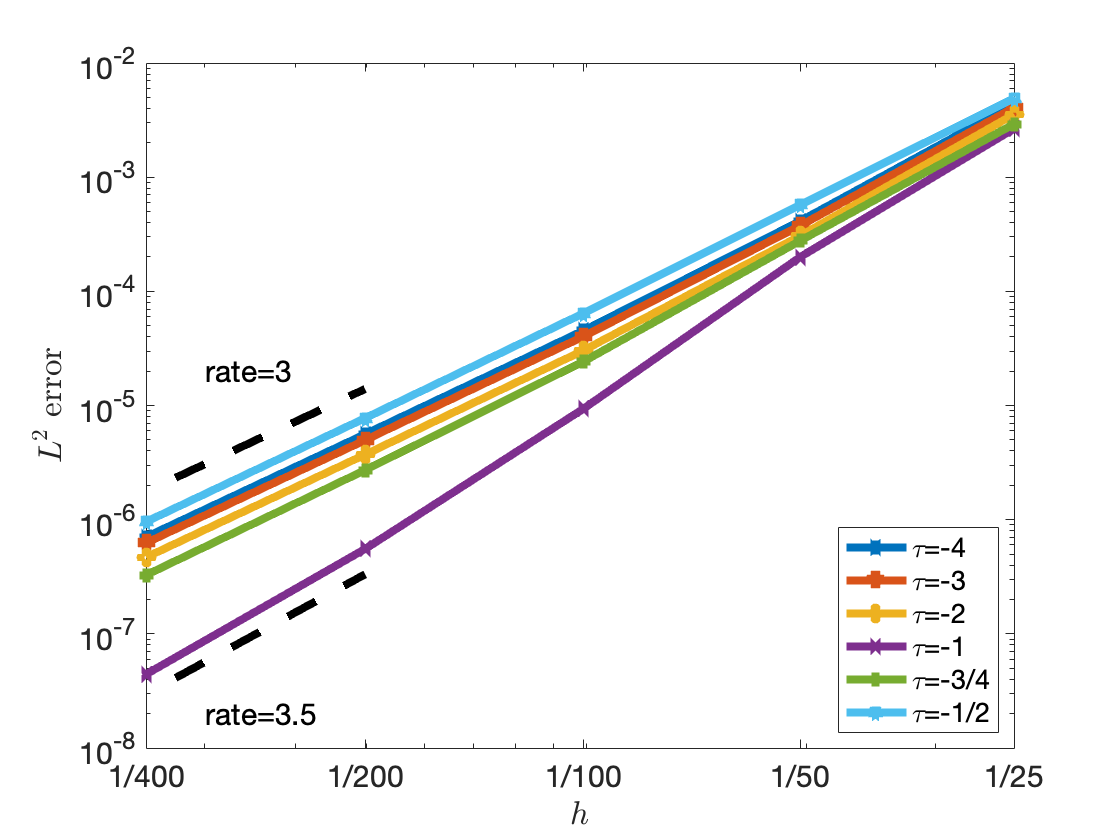}
\caption{$L^2$ error of the SBP-WENO scheme for $p=3$ (left) and 4 (right) with different values of $\tau$ for the advection equation.}
\label{Ex2_order3_smooth}
\end{figure}

Next, we consider the advection equation in domain $[-1,1]$ with nonsmooth solution. We choose the initial condition to be zero, and the inflow boundary condition 
\begin{align*}
    u(-1, t) = \begin{cases}
        \left[e^{-\beta (-1-t-z_{-})^2} + e^{-\beta (-1-t-z_{+}))^2} + 4e^{-\beta (-1-t-z)^2}\right]/6, & 0\leq t\leq 0.4,\\
        1, & 0.6\leq t\leq 0.8, \\
        1 - |10(-1-t+2.1)|, & 1\leq t\leq 1.2, \\
        \left[\sqrt{ \max(0, 1-\alpha_{-}^2(-1-t+2.5)^2)} + \right.\\
        \sqrt{ \max(0, 1-\alpha_{+}^2(-1-t+2.5)^2)} + \\
        \left.4\sqrt{ \max(0, 1-\alpha^2(-1-t+2.5)^2)}\right]/6, & 1.4\leq t\leq 1.6, \\
        0, & \text{elsewhere},
    \end{cases}
\end{align*}
where $\delta_z = 0.005$, $z=-1.2$, $z_{-} = z-\delta_z$, $z_{+} = z+\delta_z$, $\alpha=10$, $\alpha_{-}=\alpha-\delta_z$, $\alpha_{+}=\alpha+\delta_z$, $\beta = \log(2)/(36\delta_z^2)$. 
The numerical solution and exact solution at time $t=1.9$ is plotted in Figure \ref{Ex3_4shapes}.
\begin{figure}
\centering
\includegraphics[width=0.49\textwidth]{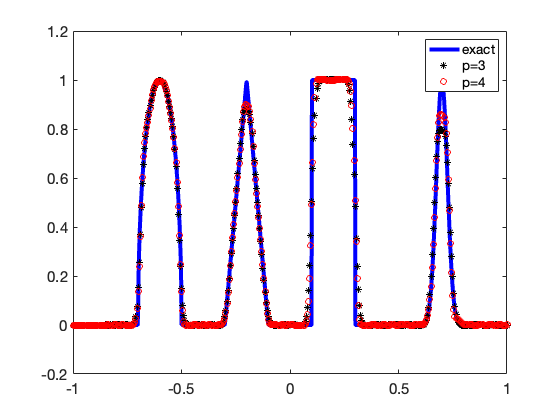}
\includegraphics[width=0.49\textwidth]{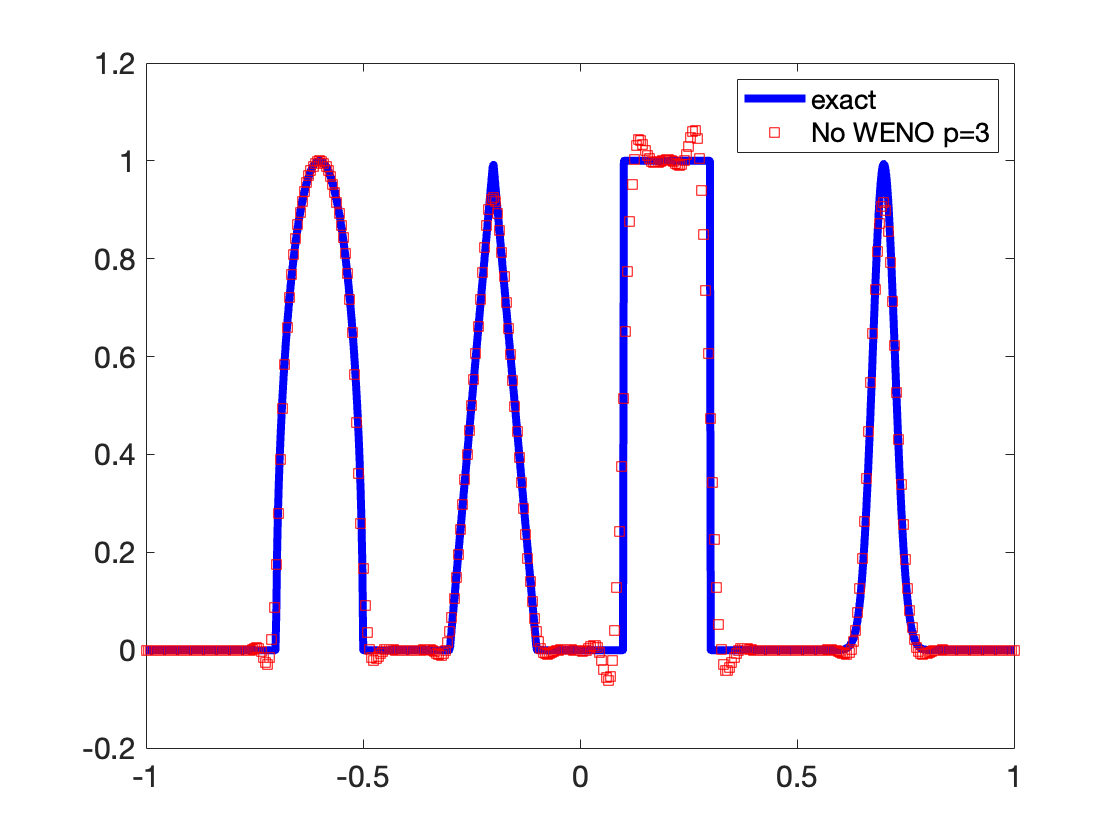}
\caption{Solutions computed with 401 grid points. Left: SBP-WENO scheme with $p=3$ and 4. {\color{black}Right: SBP scheme without WENO.}}
\label{Ex3_4shapes}
\end{figure}
We observe that the numerical solution agrees very well with the exact solution for the SBP-WENO discretization with $p=3$ and 4. On the other hand, the result without the WENO scheme has oscillations near the nonsmooth region.

\section{Conclusion}

In an SBP-SAT finite difference discretization, the truncation error is often larger on a few grid points near the boundaries than in the interior. It is well-known that for first order hyperbolic PDEs, the overall convergence rate is one order higher than the order of truncation error near the boundaries \cite{Gustafsson1981}. However, with an  upwind SBP-SAT method, it was reported that the convergence rate is higher than expected \cite{Mattsson2017}. Using the normal mode analysis, we prove that the convergence rate in an upwind SBP-SAT method  is indeed higher than the results in \cite{Gustafsson1981}. More precisely, we have considered a scheme with third order truncation error in the interior and first order truncation error on a few grid points near boundaries. With a particular choice of the penalty parameters, the convergence rate is 1.5 orders higher than the boundary truncation error, and all other stable choices of the penalty parameters lead to a convergence rate that is one order higher than the boundary truncation error. Both first order hyperbolic equation and hyperbolic systems are considered. We have carried out numerical experiments with higher order discretizations and observed that the penalty parameters play the same role in convergence rate.   

When using an SBP-SAT finite difference method to solve PDEs with nonsmooth data, the numerical solution contains oscillations near the region where the solution is not sufficiently smooth. Using the WENO methodology, we derive an SBP-SAT discretization that can resolve nonsmooth solutions very well, and converges to optimal order for smooth solutions. In addition, the discretization remains energy stable. 

\section*{Acknowledgement}
The authors acknowledge the financial support from STINT (Project IB2019-8542). Yan Jiang is partially supported by NSFC grant 12271499 and Cyrus Tang Foundation.

\appendix
\section*{Appendix: WENO formulas near boundaries} \label{sec:WENO_bound}

\begin{itemize}

\item At $\bar x_0=x_1$: $\hat{u}_0 = u_1$

\item At $\bar x_1=\frac{49}{149}$:
\begin{align*}
\hat{u}_1 
= \frac{25}{48} u_1 + \frac{168}{288} u_2 - \frac{11}{144} u_3 -\frac{1}{32}u_4 
:= d^{(1)}_1 \hat u^{(1)}_1 + d^{(2)}_1 \hat u^{(2)}_1 
\end{align*}
with 
\begin{align*}
    & d^{(1)}_1= \frac{15}{19}, \quad d^{(2)}_{1} = \frac{4}{19}, \\
    & \hat u^{(1)}_1 = \frac{95}{144} u_1 + \frac{49}{144} u_2,\\
    & \hat u^{(2)}_1 = \frac{1741}{1152}u_2 - \frac{209}{576}u_3 - \frac{19}{128} u_4.
\end{align*}
The nonlinear weights are given as
\begin{align*}
    & \beta_1 = (u_2 - u_1)^2 \\
    & \beta_2 = \frac{1}{16}(9u_2 -14u_3 + 5u_4)^2 + \frac{49}{48}(u_2 - 2u_3 +u_4)^2 \\
    &\tau = |\beta_2 -\beta_1|^2, \\
    & w_{1}^{(r)} = \frac{\alpha_{1}^{(r)}}{\alpha_{1}^{(1)} +\alpha_{1}^{(2)}}, \quad
    \alpha_1^{(r)} = d^{(r)}_1 (1+\frac{\tau}{\beta_r + \epsilon}), \quad r=1, 2.
\end{align*}

\item  At  $\bar x_2 = \frac{29}{18}h$:\begin{align*}
\hat{u}_2 
= -\frac{55}{288} u_1 + \frac{235}{288} u_2 + \frac{95}{288} u_3 + \frac{13}{288}u_4 
:= d^{(1)}_2 \hat u^{(1)}_2 + d^{(2)}_2 \hat u^{(2)}_2 
\end{align*}
with 
\begin{align*}
    & d^{(1)}_2= \frac{5}{16}, \quad d^{(2)}_{2} = \frac{11}{16}, \\
    & \hat u^{(1)}_2 = -\frac{11}{18} u_1 + \frac{29}{18} u_2,\\
    & \hat u^{(2)}_2 = \frac{5}{11}u_2 + \frac{95}{198}u_3 + \frac{13}{198} u_4.
\end{align*}
The nonlinear weights are given as
\begin{align*}
    & \beta_1 = (u_2 - u_1)^2 \\
    & \beta_2 = \frac{13}{12}(u_2 -2u_3 + u_4)^2 + \frac{1}{4}(3u_2 - 4u_3 +u_4)^2 \\
    &\tau = |\beta_2 -\beta_1|, \\
    & w_{2}^{(r)} = \frac{\alpha_{2}^{(r)}}{\alpha_{2}^{(1)} +\alpha_{2}^{(2)}}, \quad
    \alpha_2^{(r)} = d_2^{(r)} (1+\frac{\tau}{\beta_r + \epsilon}), \quad r=1, 2.
\end{align*}

\item At  $\bar x_3 = \frac{355}{144}h$:\begin{align*}
\hat{u}_3 
= \frac{1}{96} u_1 - \frac{31}{144} u_2 + \frac{269}{288} u_3 + \frac{13}{48}u_4 
:= d^{(1)}_3 \hat u^{(1)}_3 + d^{(2)}_3 \hat u^{(2)}_3 + d^{(3)}_3 \hat u^{(3)}_3 
\end{align*}
with 
\begin{align*}
    & d^{(1)}_3= \frac{39}{67}, \quad d^{(2)}_{3} = \frac{689}{1809}, \quad d^{(3)}_{3} = \frac{1}{27}, \\
    & \hat u^{(1)}_3 = \frac{77}{144} u_3 + \frac{67}{144} u_4,\\
    & \hat u^{(2)}_3 = -\frac{67}{144} u_2 + \frac{211}{144} u_3,\\
    & \hat u^{(3)}_3 = \frac{9}{32}u_1 - \frac{37}{36}u_2 + \frac{503}{288} u_3.
\end{align*}
The nonlinear weights are given as
\begin{align*}
    & \beta_1 = (u_4 - u_3)^2, \quad \beta_2 = (u_3 - u_2)^2 \\
    & \beta_3 = \frac{13}{12}(u_1 -2u_2 + u_3)^2 + \frac{1}{4}(u_1 - 4u_2 + 3u_3)^2 \\
    &\tau = |\beta_1 + \beta_2 -2\beta_3|, \\
    & w_{3}^{(r)} = \frac{\alpha_{3}^{(r)}}{\alpha_{3}^{(1)} +\alpha_{3}^{(2)}+\alpha_{3}^{(3)}}, \quad
    \alpha_3^{(r)} = d_3^{(r)} (1+\frac{\tau}{\beta_r + \epsilon}), \quad r=1, 2, 3.
\end{align*}

\item At  $\bar x_n = x_n$: $\hat{u}_{n}=u_n$

\item At $\bar x_{n-1}=x_n - \frac{49}{144}h$: 
\begin{align*}
\hat{u}_{n-1} 
= \frac{23}{48} u_n + \frac{205}{288} u_{n-1} - \frac{29}{144} u_{n-2} + \frac{1}{96}u_{n-3} 
:= d^{(1)}_{n-1} \hat u^{(1)}_{n-1} + d^{(2)}_{n-1} \hat u^{(2)}_{n-1} + d^{(3)}_{n-1} \hat u^{(3)}_{n-1} 
\end{align*}
with 
\begin{align*}
    & d^{(1)}_{n-1}= \frac{69}{95}, \quad d^{(2)}_{n-1} = \frac{1127}{4465}, \quad d^{(3)}_{n-1} = \frac{1}{47}, \\
    & \hat u^{(1)}_{n-1} = \frac{95}{144} u_n + \frac{49}{144} u_{n-1},\\
    & \hat u^{(2)}_{n-1} = \frac{239}{144} u_{n-1} - \frac{95}{144} u_{n-2},\\
    & \hat u^{(3)}_{n-1} = \frac{619}{288} u_{n-1} - \frac{59}{36} u_{n-2} + \frac{47}{96} u_{n-3}.
\end{align*}
The nonlinear weights are given as
\begin{align*}
    & \beta_1 = (u_n - u_{n-1})^2, \quad \beta_2 = (u_{n-1} - u_{n-2})^2 \\
    & \beta_3 = \frac{13}{12}(u_{n-1} -2u_{n-2} + u_{n-3})^2 + \frac{1}{4}(3u_{n-1} - 4u_{n-2} + u_{n-3})^2 \\
    &\tau = |\beta_1 + \beta_2 -2\beta_3|, \\
    & w_{n-1}^{(r)} = \frac{\alpha_{n-1}^{(r)}}{\alpha_{n-2}^{(1)} +\alpha_{n-1}^{(2)}+\alpha_{n-1}^{(3)}}, \quad
    \alpha_{n-1}^{(r)} = d_{n-1}^{(r)} (1+\frac{\tau}{\beta_r + \epsilon}), \quad r=1, 2, 3.
\end{align*}

\item At $\bar x_{n-2}=x_n - \frac{29}{18}h$: 
\begin{align*}
\hat{u}_{n-2} 
&= -\frac{31}{288} u_n + \frac{139}{288} u_{n-1} + \frac{239}{288} u_{n-2} - \frac{83}{288}u_{n-3} + \frac{1}{12} u_{n-4} \\
&:= d^{(1)}_{n-2} \hat u^{(1)}_{n-2} + d^{(2)}_{n-2} \hat u^{(2)}_{n-2} + d^{(3)}_{n-2} \hat u^{(3)}_{n-2} 
\end{align*}
with 
\begin{align*}
    & d^{(1)}_{n-2}= \frac{31}{197}, \quad d^{(2)}_{n-2} = \frac{45}{88}, \quad d^{(3)}_{n-2} = \frac{5}{16}, \\
    & \hat u^{(1)}_{n-2} = -\frac{11}{18} u_n + \frac{29}{18} u_{n-1},\\
    & \hat u^{(2)}_{n-2} = \frac{7}{18} u_{n-1} + \frac{11}{18} u_{n-2},\\
    & \hat u^{(3)}_{n-2} = \frac{149}{90} u_{n-2} - \frac{83}{90} u_{n-3} + \frac{24}{90} u_{n-4}.
\end{align*}
The nonlinear weights are given as
\begin{align*}
    & \beta_1 = (u_n - u_{n-1})^2, \quad \beta_2 = (u_{n-1} - u_{n-2})^2 \\
    & \beta_3 = \frac{13}{12}(u_{n-2} -2u_{n-3} + u_{n-4})^2 + \frac{1}{4}(3u_{n-2} - 4u_{n-3} + u_{n-4})^2 \\
    &\tau = |\beta_1 + \beta_2 -2\beta_3|, \\
    & w_{n-2}^{(r)} = \frac{\alpha_{n-2}^{(r)}}{\alpha_{n-2}^{(1)} +\alpha_{n-2}^{(2)}+\alpha_{n-2}^{(3)}}, \quad
    \alpha^{(r)}_{n-2} = d_{n-2}^{(r)} (1+\frac{\tau}{\beta_r + \epsilon}), \quad r=1, 2, 3.
\end{align*}

\item At $\bar x_{n-3}=x_n - \frac{355}{144}h$: 
\begin{align*}
\hat{u}_{n-3} 
=& -\frac{1}{32} u_n + \frac{11}{144} u_{n-1} + \frac{65}{288} u_{n-2} + \frac{17}{16}u_{n-3} - \frac{5}{12} u_{n-4} + \frac{1}{12} u_{n-5} \\ 
:=& d^{(1)}_{n-3} \hat u^{(1)}_{n-3} + d^{(2)}_{n-3} \hat u^{(2)}_{n-3} + d^{(3)}_{n-3} \hat u^{(3)}_{n-3} 
\end{align*}
with 
\begin{align*}
    & d^{(1)}_{n-3}= \frac{5}{53}, \quad d^{(2)}_{n-3} = \frac{3576}{5353}, \quad d^{(3)}_{n-3} = \frac{24}{101}, \\
    & \hat u^{(1)}_{n-3} = -\frac{53}{160} u_n + \frac{583}{720} u_{n-1} - \frac{131}{1440} u_{n-2} + \frac{49}{80} u_{n-3},\\
    & \hat u^{(2)}_{n-3} = \frac{101}{288} u_{n-2} + \frac{5}{6} u_{n-3} - \frac{53}{288} u_{n-4},\\
    & \hat u^{(3)}_{n-3} = \frac{181}{96} u_{n-3} - \frac{89}{72} u_{n-4} + \frac{101}{288} u_{n-5}.
\end{align*}
The nonlinear weights are given as
\begin{align*}
    & \beta_1 = \frac{781}{720}(u_n - 3u_{n-1} +3u_{n-2} - u_{n-3})^2 + \frac{13}{12}(u_n - 4u_{n-1} + 5u_{n-2} - 2u_{n-3})^2 
\\
& \qquad + \frac{1}{64}(3u_n - 13u_{n-1} + 25 u_{n-2} - 15u_{n-3})^2 , \\
    & \beta_2 = \frac{13}{12}(u_{n-2} -2u_{n-3} + u_{n-4})^2 + \frac{1}{4}(u_{n-2} - u_{n-4})^2 \\
    & \beta_3 = \frac{13}{12}(u_{n-3} -2u_{n-4} + u_{n-5})^2 + \frac{1}{4}(3u_{n-3} - 4u_{n-4} + u_{n-5})^2 \\
    &\tau = |4\beta_1 -3 \beta_2 - \beta_3|, \\
    & w_{n-3}^{(r)} = \frac{\alpha_{n-3}^{(r)}}{\alpha_{n-3}^{(1)}+\alpha_{n-3}^{(2)}+\alpha_{n-3}^{(3)}}, \quad
    \alpha_{n-3}^{(r)} = d_{n-3}^{(r)} (1+\frac{\tau}{\beta_r + \epsilon}), \quad r=1, 2, 3.
\end{align*}

\end{itemize}

\bibliographystyle{plain} 
\bibliography{Siyang_References} 

\begin{thebibliography}{10}

\bibitem{Almquist2019}
M.~Almquist, S.~Wang, and J.~Werpers.
\newblock Order-preserving interpolation for summation-by-parts operators at
  nonconforming grid interfaces.
\newblock {\em SIAM J. Sci. Comput.}, 41:A1201--A1227, 2019.

\bibitem{Carpenter1994}
M.~H. Carpenter, D.~Gottlieb, and S.~Abarbanel.
\newblock Time--stable boundary conditions for finite--difference schemes
  solving hyperbolic systems: methodology and application to high--order
  compact schemes.
\newblock {\em J. Comput. Phys.}, 111:220--236, 1994.

\bibitem{Carpenter1999}
M.~H. Carpenter, J.~Nordstr\"{o}m, and D.~Gottlieb.
\newblock A stable and conservative interface treatment of arbitrary spatial
  accuracy.
\newblock {\em J. Comput. Phys.}, 148:341--365, 1999.

\bibitem{Del2014Review}
D.~C. {Del Rey Fern\'{a}ndez}, J.~E. Hicken, and D.~W. Zingg.
\newblock Review of summation--by--parts operators with simultaneous
  approximation terms for the numerical solution of partial differential
  equations.
\newblock {\em Comput. Fluids}, 95:171--196, 2014.

\bibitem{Fisher2011}
T.~C. Fisher, M.~H. Carpenter, N.~K. Yamaleev, and S.~H. Frankel.
\newblock Boundary closures for fourth-order energy stable weighted essentially
  non-oscillatory finite-difference schemes.
\newblock {\em J. Comput. Phys.}, 230:3727--3752, 2011.

\bibitem{Gustafsson1975}
B.~Gustafsson.
\newblock The convergence rate for difference approximations to mixed initial
  boundary value problems.
\newblock {\em Math. Comput.}, 29:396--406, 1975.

\bibitem{Gustafsson1981}
B.~Gustafsson.
\newblock The convergence rate for difference approximations to general mixed
  initial boundary value problems.
\newblock {\em SIAM. J. Numer. Anal.}, 18:179--190, 1981.

\bibitem{Gustafsson2013}
B.~Gustafsson, H.~O. Kreiss, and J.~Oliger.
\newblock Time--{D}ependent {P}roblems and {D}ifference {M}ethods.
\newblock {\em John Wiley \& Sons}, 2013.

\bibitem{Hagstrom2012}
T.~Hagstrom and G.~Hagstrom.
\newblock Grid stabilization of high--order one--sided differencing {II}:
  second--order wave equations.
\newblock {\em J. Comput. Phys.}, 231:7907--7931, 2012.

\bibitem{Hicken2016}
J.~E. Hicken, D.~C. {Del Rey Fern\'{a}ndez}, and D.~W. Zingg.
\newblock Multidimensional summation--by--parts operators: general theory and
  application to simplex elements.
\newblock {\em SIAM J. Sci. Comput.}, 38:A1935--A1958, 2016.

\bibitem{Jiang1996}
Y.~Jiang and C.~W. Shu.
\newblock Efficient implementation of weighted {ENO} schemes.
\newblock {\em J. Comput. Phys.}, 126:202--228, 1996.

\bibitem{Kreiss1972}
H.~O. Kreiss and J.~Oliger.
\newblock {C}omparison of accurate methods for the integration of hyperbolic
  equations.
\newblock {\em Tellus}, 24:199--215, 1972.

\bibitem{Kreiss1974}
H.~O. Kreiss and G.~Scherer.
\newblock Finite element and finite difference methods for hyperbolic partial
  differential equations.
\newblock {\em Mathematical Aspects of Finite Elements in Partial Differential
  Equations, Symposium Proceedings}, pages 195--212, 1974.

\bibitem{Liu1994}
X.~D. Liu, S.~Osher, and T.~Chan.
\newblock Weighted essentially nonoscillatory schemes.
\newblock {\em J. Comput. Phys.}, 115, 1994.

\bibitem{Ludvigsson2018}
G.~Ludvigsson, K.~R. Steffen, S.~Sticko, S.~Wang, Q.~Xia, Y.~Epshteyn, and
  G.~Kreiss.
\newblock High-order numerical methods for 2{D} parabolic problems in single
  and composite domains.
\newblock {\em J. Sci. Comput.}, 76, 2018.

\bibitem{Mattsson2017}
K.~Mattsson.
\newblock Diagonal--norm upwind {SBP} operators.
\newblock {\em J. Comput. Phys.}, 335:283--310, 2017.

\bibitem{Nordstrom1999}
J.~Nordstr\"{o}m and M.~H. Carpenter.
\newblock Boundary and interface conditions for high-order finite-difference
  methods applied to the {E}uler and {N}avier-{S}tokes equations.
\newblock {\em J. Comput. Phys.}, 148:621--645, 1999.

\bibitem{Svard2014}
M.~Sv\"{a}rd and J.~Nordstr\"{o}m.
\newblock Review of summation--by--parts schemes for initial--boundary--value
  problems.
\newblock {\em J. Comput. Phys.}, 268:17--38, 2014.

\bibitem{Wang2017}
S.~Wang and G.~Kreiss.
\newblock Convergence of summation--by--parts finite difference methods for the
  wave equation.
\newblock {\em J. Sci. Comput.}, 71:219--245, 2017.

\bibitem{Yamaleev2009b}
N.~K. Yamaleev and M.~H. Carpenter.
\newblock A systematic methodology for constructing high-order energy stable
  {WENO} schemes.
\newblock {\em J. Comput. Phys.}, 228:4248--4272, 2009.

\bibitem{Yamaleev2009}
N.~K. Yamaleev and M.~H. Carpenter.
\newblock Third-order energy stable {WENO} scheme.
\newblock {\em J. Comput. Phys.}, 228:3025--3047, 2009.

\bibitem{ZHU2018659}
J.~Zhu and C.-W. Shu.
\newblock A new type of multi-resolution weno schemes with increasingly higher
  order of accuracy.
\newblock {\em J. Comput. Phys.}, 375:659--683, 2018.

\bibitem{ZHU201919}
J.~Zhu and C.-W. Shu.
\newblock A new type of multi-resolution weno schemes with increasingly higher
  order of accuracy on triangular meshes.
\newblock {\em J. Comput. Phys.}, 392:19--33, 2019.

\bibitem{ZHU2020}
J.~Zhu and C.-W. Shu.
\newblock Numerical study on the convergence to steady state solutions of a new
  class of high order weno schemes.
\newblock {\em Comm. App. Math. Comp.}, 2:429–460, 2020.

\end{thebibliography}

\end{document}